\documentclass[hidelinks,onefignum,onetabnum]{siamart251216}

\usepackage{lipsum}
\usepackage{amsfonts}
\usepackage{graphicx}
\usepackage{epstopdf}
\usepackage{algorithmic}
\ifpdf
  \DeclareGraphicsExtensions{.eps,.pdf,.png,.jpg}
\else
  \DeclareGraphicsExtensions{.eps}
\fi

\usepackage{algorithm}
\usepackage{cases}
\usepackage{hyperref}
\usepackage{cleveref}
\usepackage{geometry}
\usepackage{booktabs}

\newsiamremark{remark}{Remark}
\newsiamremark{hypothesis}{Hypothesis}
\crefname{hypothesis}{Hypothesis}{Hypotheses}
\newsiamthm{claim}{Claim}
\newsiamthm{asmp}{Assumption}
\newsiamremark{fact}{Fact}
\crefname{fact}{Fact}{Facts}

\newcommand{\R}{\mathbb{R}}

\newcommand{\KL}{\mathrm{KL}}

\renewcommand{\>}{\rangle}
\newcommand{\ones}{\mathbf{1}}

\newcommand{\Proj}{\mathrm{Proj}}

\newsiamthm{prob}{Problem}

\headers{Multi-Observation Inverse Optimal Transport}{Chenglong Bao, Zanyu Li and Defeng Sun}

\title{Learning Universal Costs via Multi-Observation Inverse Optimal Transport\thanks{Submitted to the editors DATE.
\funding{This research was supported by the National Natural Science Foundation of China (No.12625119, W2612010). Defeng Sun was supported by the Research Center for Intelligent Operations Research and RGC Senior Research Fellow Scheme No. SRFS2223-5S02.}}}

\author{Chenglong Bao\thanks{Corresponding Author. Yau Mathematical Sciences Center, Tsinghua University, Beijing 100084, China, 
  \email{clbao@tsinghua.edu.cn}.}
\and Zanyu Li\thanks{Department of Applied Mathematics, The Hong Kong Polytechnic University, Hong Kong, China,
  \email{zanyu.li@polyu.edu.hk}.}
\and Defeng Sun\thanks{Department of Applied Mathematics, The Hong Kong Polytechnic University, Hong Kong, China, \email{defeng.sun@polyu.edu.hk}.}
}

\usepackage{amsopn}
\DeclareMathOperator{\diag}{diag}

\ifpdf
\hypersetup{
  pdftitle={Learning Universal Costs via Multi-Observation Inverse Optimal Transport},
  pdfauthor={Chenglong Bao, Zanyu Li, and Defeng Sun}
}
\fi

\begin{document}

\maketitle

\begin{abstract}
Classical Inverse Optimal Transport (IOT) typically infers a cost function from a single observed transport plan, which may be sensitive to noise and observational biases in practice. To address this limitation, we propose the Multi-Observation Inverse Optimal Transport (M-IOT) framework, which learns a universal cost function from multiple transport plans. We reformulate the resulting bilevel optimization problem into a tractable single-level convex program and characterize the existence and geometric structure of the optimal solution set. Moreover, we propose a Block Coordinate Descent with Anderson-type Extrapolation (BCDwAE) algorithm and establish its Q-linear convergence. Extensive experiments on synthetic data, urban mobility flows, and e-commerce recommendations demonstrate that M-IOT reduces recovery error compared to single-observation baselines and extracts robust structural patterns from noisy multi-view data. 
\end{abstract}



\section{Introduction}
\label{sec:introduction}

Optimal Transport (OT)~\cite{villani2008optimal} aligns probability distributions and has applications ranging from economics to machine learning \cite{genevay2019entropy,kolouri2017optimal,courty2016optimal}. While much of the OT literature focuses on the forward problem, i.e. finding the optimal transport plan given a known cost, many real-world scenarios present the inverse challenge: the cost function itself is unknown and should be inferred from observed matching patterns \cite{shi2023understanding,wang2023self,persiianov2024inverse,liu2019learning,yu2022explainable,dupuy2019estimating,samaran2024scconfluence}. The Inverse Optimal Transport (IOT) problem arises naturally in applications such as immigration flows \cite{stuart2020inverse}, marriage matching \cite{li2019learning}, and legal case analysis \cite{yu2022explainable}. In these settings, the goal is to infer the underlying structure that gives rise to the observed correlations. 

Since its formalization \cite{stuart2020inverse,li2019learning}, IOT has developed along several key directions. From the theoretical perspective, several studies \cite{10.1088/1361-6420/ae5ace,gonzalez2024nonlinear,chiu2022discrete,gonzalez2024identifiability} investigate the well-posedness of IOT problem. From a modeling standpoint ,   \cite{stuart2020inverse,chiu2022discrete} formulate inverse optimal transport in bayesian framework, while \cite{ma2020learning,dupuy2019estimating} cast it as a bilevel optimization problem with an equivalent convex single level formulation. As \cite{dupuy2019estimating} pointed out, this convex single level formulation corresponds to a maximum likelihood estimation (MLE) problem and has been widely studied in \cite{andrade2023sparsistency,10.1088/1361-6420/ae5ace,carlier2023sista}.  Based on the MLE framework, \cite{andrade2025learning} propose a sharpened Fenchel-Young loss and investigate the inverse unbalanced optimal transport problem, while \cite{10.1088/1361-6420/ae5ace} generalize the approach to Bregman regularized inverse optimal transport and analyze the conditions for existence and uniqueness of the optimal solution. Several algorithms have been proposed to solve the resulting convex IOT problem, including block coordinate descent methods \cite{ma2020learning,10.1088/1361-6420/ae5ace,carlier2023sista}, quasi-Newton method \cite{andrade2023sparsistency} and neural network method \cite{ma2020learning,liu2019learning}. Notably, alternative IOT formulations tailored to different application settings have also been developed \cite{elvander2025mixtures,mascherpa2025convex}, further expanding the scope of IOT.  The modeling capability of IOT has been demonstrated across diverse domains, including legal case analysis\cite{yu2022explainable}, marriage matching\cite{dupuy2019estimating},  biology\cite{samaran2024scconfluence,elvander2025mixtures} and machine learning\cite{shi2023understanding,wang2023self,persiianov2024inverse,liu2019learning}.

Despite these advances, the existing IOT literature typically adopts a single-observation paradigm, in which the cost function is inferred from a single transport plan. Across theoretical studies on well-posedness \cite{10.1088/1361-6420/ae5ace}, convex reformulations \cite{andrade2023sparsistency,andrade2025learning}, and application-driven work \cite{dupuy2019estimating,samaran2024scconfluence}, the prevailing approach relies on a single snapshot of data to recover the cost. However, this approach can be fragile: real-world observations are often noisy, context-dependent, and provide only a partial view of the underlying system. Moreover, for systems that evolve continuously over time, such as urban traffic flows or cellular transport processes, it is unrealistic to expect that an invariant underlying structure can be reliably identified from a single snapshot. 
Although a small number of recent works have begun to consider multiple observations, for example in theoretical studies of identifiability \cite{gonzalez2024identifiability,gonzalez2024nonlinear}, they do not provide a unified optimization framework for jointly recovering a shared, invariant cost structure from multiple observations. Consequently, a gap remains between the practical need for robustness under multiple observations and the current theoretical and algorithmic foundations of IOT methods.

To bridge this gap, we introduce the Multi-Observation Inverse Optimal Transport (M-IOT) framework, which aims to infer a universal cost by jointly learning from multiple transport plans. We reformulate the inverse problem, originally posed as a bilevel optimization that minimizes an aggregated Kullback-Leibler divergence, into an equivalent, tractable single-level convex optimization problem. Firstly, we characterize the existence of solutions and analyze the structure of the solution set. Secondly, we propose a Block Coordinate Descent with Anderson-type Extrapolation (BCDwAE) algorithm to solve the model, and establish its linear convergence. Finally, we evaluate the proposed framework on synthetic data, urban mobility flows, and e-commerce recommendation tasks. These experiments demonstrate that M-IOT effectively recovers robust cost structures from noisy, multi-view data and outperforms the single-observation model.


\section{Notations and Preliminaries}
\label{sec:preliminaries}

In this section, we summarize the notations used throughout the paper and review some fundamental concepts from convex analysis and optimal transport.

We use uppercase letters (e.g., $C, X$) for matrices and lowercase letters (e.g., $u, v, z$) for vectors. The set of $n \times n$ matrices with non-negative entries is denoted by $\mathbb{R}^{n \times n}_{+}$ and the set of $n \times n$ matrices with strictly positive entries is denoted by $\mathbb{R}^{n\times n}_{++}$. The vector of all ones is denoted by $\mathbf{1}_n$. The Frobenius inner product between two matrices $A, B \in \mathbb{R}^{n \times n}$ is $\langle A, B \rangle = \sum_{i,j} A_{ij}B_{ij}$. The vectorization of a matrix $C$ is denoted by $\text{vec}(C)$. 
We use $\ker(A)$ to denote the null space of a linear operator $A$, and $\ker(A)^\perp$ for its orthogonal complement. The orthogonal projection onto a subspace $\mathcal{W}$ is denoted by $\Pi_\mathcal{W}$. The symbol $\oplus$ denotes the direct sum of block variables, and $N_{\mathcal C}(C)$ denotes the normal cone to a closed convex set $\mathcal C$ at $C\in\mathcal C$.

Let $\mathbb{E}$ be a finite-dimensional Euclidean space equipped with the inner product $\langle \cdot, \cdot \rangle$ and the induced norm $\|\cdot\|$. For a non-empty closed convex set $\mathcal{D} \subseteq \mathbb{E}$, the distance from a point $x \in \mathbb{E}$ to $\mathcal{D}$ is defined as $\text{dist}(x, \mathcal{D}) := \min_{y \in \mathcal{D}} \|x - y\|$. The projection onto a closed set $\mathcal{D}$ is denoted by $\text{Proj}_\mathcal{D}$ and is defined by $\text{Proj}_\mathcal{D}(x):=\arg\min_{y\in\mathcal{D}}\frac{1}{2}\|y-x\|^2$. We denote the relative interior of a convex set $\mathcal{D}$ by $\text{ri}(\mathcal{D})$. For any proper closed convex function $f:\mathbb{E}\rightarrow (-\infty,+\infty]$, define the sublevel set of $f$ to be $\mathcal{S}_\alpha^f:=\{x\in\mathbb{E}: f(x) \leq \alpha\}$. Given a nonempty convex set $\mathcal{C}$, define the recession cone of $\mathcal{C}$ to be $R_\mathcal{C}:=\{d:x+td\in \mathcal{C}, \forall x\in \mathcal{C} , \forall t > 0\}$ and denote $L_\mathcal{C}:= R_\mathcal{C}\cap (-R_\mathcal{C})$ the lineality space of $\mathcal{C}$. For the sake of completeness, we list some properties of the recession cone used in this article. For more detailed analysis of the recession cone, we refer the interested readers to the classic convex analysis textbooks \cite{rockafellar1997convex,bertsekas2009convex}.

Let $\mu \in \R^n_+$ and $\nu \in \R^n_+$ be two discrete probability distributions over a set of $n$ locations, satisfying $\sum_{i=1}^n \mu_i = \sum_{j=1}^n \nu_j = 1$. Let $C \in \R^{n \times n}$ be a cost matrix where $C_{ij}$ denotes the cost of transporting a unit of mass from location $i$ to location $j$. The set of feasible transport plans between $\mu$ and $\nu$ is the transport polytope, defined as:
\begin{equation}
    \mathcal{U}(\mu, \nu) := \{ X \in \R^{n \times n}_+ \mid X \ones_n = \mu, X^T \ones_n = \nu \},
\end{equation}
where $\ones_n$ is the $n$-dimensional vector of all ones. Given $\gamma>0$, the entropy-regularized OT problem seeks the transport plan $X \in \mathcal{U}(\mu, \nu)$ by minimizing
\begin{equation}
    \label{eq:forward_ot}
    L_\gamma(C, \mu, \nu) := \min_{X \in \mathcal{U}(\mu, \nu)} \<C, X\> - \gamma H(X),
\end{equation}
where $H(X) = -\sum_{i,j} X_{ij}(\log X_{ij} - 1)$ is the entropy of the transport plan $X$. 

Throughout this paper, $\{\hat{X}^l\}_{l=1}^K$ denotes the set of observed transport plans, and we assume $\hat{X}^l\in\mathbb{R}^{n\times n}_{++}$. Zero-valued data entries are removed during preprocessing, so the observation matrices supplied to the model are strictly positive. We denote $\mu^l = \hat{X}^l\ones_n$ and $\nu^l = (\hat{X}^l)^T\ones_n$.

\section{Problem Formulation}
\label{sec:formulation}

In this section, we develop the M-IOT framework. We begin by analyzing the theoretical properties of the inverse problem and demonstrate its ill-posedness in practical settings. We then introduce a robust optimization framework, first formulated as a bilevel program and subsequently reformulated as an equivalent, tractable single-level convex program.

\subsection{The  Inverse Problem and Its Ill-Posedness}
\label{sec:ideal_problem}

Before constructing our optimization model, we first examine the well-posedness of the inverse problem. This problem aims to identify a single, universal cost matrix $C$ that can be determined by a collection of observations $\{ (\hat{X}^l, \mu^l, \nu^l) \}_{l=1}^K$. Formally, we seek a matrix $C$ such that
\begin{equation}
    \forall l \in \{1, \dots, K\}, \quad X^*(C, \mu^l, \nu^l) = \hat{X}^l,
\end{equation}
where $$
X^*(C, \mu^l, \nu^l) := \arg\min_{X \in \mathcal{U}(\mu^l, \nu^l)} \<C, X\> - \gamma H(X).
$$

For a single observation $l$, the cost matrices satisfying $X^*(C, \mu^l, \nu^l) = \hat{X}^l$ form an affine subspace $\mathcal{S}_l$. For $\hat{X}^l \in \mathbb{R}^{n \times n}_{++}$, this solution set is characterized in \cite{10.1088/1361-6420/ae5ace,chiu2022discrete} as
\begin{equation}
    \mathcal{S}_l = \{ \hat{C}^l + A \mid A \in \mathcal{V} \},
\end{equation}
where $\hat{C}^l$ is  defined by $(\hat{C}^l)_{ij} = -\gamma \log \hat{X}^l_{ij}$, and $\mathcal{V} := \{ A \in \mathbb{R}^{n \times n} \mid A_{ij} = a_i + b_j \text{ for some } a, b \in \mathbb{R}^n \}$ is a $(2n-1)$-dimensional vector space. This characterization implies that all solution sets $\mathcal{S}_l$ are parallel affine subspaces, each being a translate of the same vector space $\mathcal{V}$, which severely restricts the existence of a joint solution, as formalized below.

\begin{proposition}
\label{prop:data_consistency}
Let $\hat{X}^1, \hat{X}^2 \in \mathbb{R}^{n \times n}_{++}$ be two observed transport plans. Their corresponding solution sets $\mathcal{S}_1$ and $\mathcal{S}_2$ are equal if and only if there exist positive diagonal matrices $D_\alpha = \mathrm{diag}(\alpha)$ and $D_\beta = \mathrm{diag}(\beta)$ for some $\alpha, \beta \in \mathbb{R}^n_{++}$ such that $\hat{X}^2 = D_\alpha \hat{X}^1 D_\beta$.
\end{proposition}
\begin{proof}
The affine subspaces $\mathcal{S}_1 = \hat{C}^1 + \mathcal{V}$ and $\mathcal{S}_2 = \hat{C}^2 + \mathcal{V}$ are equal if and only if  $\hat{C}^1 - \hat{C}^2 \in \mathcal{V}$. By definition of $\hat{C}^l$, we have
$$ (\hat{C}^1 - \hat{C}^2)_{ij} = (-\gamma \log X^1_{ij}) - (-\gamma \log \hat{X}^2_{ij}) = \gamma \log(\hat{X}^2_{ij}/X^1_{ij}). $$
This vector belongs to $\mathcal{V}$ if and only if there exist vectors $a, b \in \mathbb{R}^n$ such that $\gamma \log(\hat{X}^2_{ij}/\hat{X}^1_{ij}) = a_i + b_j$ for all $i,j$. Dividing by $\gamma$ and taking the exponential of both sides yields
$$ \hat{X}^2_{ij}/\hat{X}^1_{ij} = \exp(a_i/\gamma) \exp(b_j/\gamma). $$
Let $\alpha_i := \exp(a_i/\gamma)$ and $\beta_j := \exp(b_j/\gamma)$. Since $a,b$ are real vectors, $\alpha, \beta \in \mathbb{R}^n_{++}$. The condition is then equivalent to $\hat{X}^2_{ij} = \alpha_i \hat{X}^1_{ij} \beta_j$, which is the element-wise form of $\hat{X}^2 = D_\alpha \hat{X}^1 D_\beta$.  
\end{proof}

This consistency condition directly determines the existence of a solution to the inverse problem.

\begin{proposition}
\label{prop:existence_ideal}
Let $\{\hat{X}^l\}_{l=1}^K \subset \mathbb{R}^{n \times n}_{++}$ be a set of observations, and let $\mathcal{C} \subset \mathbb{R}^{n \times n}$ be a nonempty closed convex set.
\begin{enumerate}
    \item \textbf{(Existence)} A solution to the inverse problem exists if and only if both of the following conditions hold:
    \begin{enumerate}
        \item for any pair $(l,m)$, there exist positive diagonal matrices $D_\alpha = \mathrm{diag}(\alpha)$ and $D_\beta = \mathrm{diag}(\beta)$ for some $\alpha, \beta \in \mathbb{R}^n_{++}$ such that $\hat{X}^l = D_\alpha \hat{X}^m D_\beta$.
        \item The common affine subspace $\mathcal{S}_1$ has a non-empty intersection with the prior set $\mathcal{C}$.
    \end{enumerate}
    \item \textbf{(Solution Set Characterization)} If a solution exists, the set of all valid cost matrices, $\mathcal{S}_{\text{feasible}}$, is given by $\mathcal{S}_{\text{feasible}} = \mathcal{S}_1 \cap \mathcal{C}$.
\end{enumerate}
\end{proposition}
\begin{proof}
A solution $C^*$ to the inverse problem must satisfy: $C^* \in (\bigcap_{l=1}^K \mathcal{S}_l) \cap \mathcal{C}$. Existence is therefore equivalent to this intersection being non-empty.
As the sets $\{\mathcal{S}_l\}$ are parallel affine subspaces, their intersection $\bigcap_{l=1}^K \mathcal{S}_l$ is non-empty if and only if $\mathcal{S}_1 = \mathcal{S}_2 = \dots = \mathcal{S}_K$. By Proposition~\ref{prop:data_consistency}, this is equivalent to condition (1a). If condition (1a) holds, then $\bigcap_{l=1}^K \mathcal{S}_l = \mathcal{S}_1$. The full solution set is thus $\mathcal{S}_{\text{feasible}} = \mathcal{S}_1 \cap \mathcal{C}$, which proves part (2). For this set to be non-empty, condition (1b) must also hold.
\end{proof}

Proposition \ref{prop:existence_ideal} shows that the inverse problem admits an exact solution only when all observations are related by positive diagonal scaling. In practice, noise and model mismatch typically violate the consistency condition in Proposition \ref{prop:data_consistency}, making the solution sets non-intersecting and the problem ill-posed. This motivates a variational formulation that minimizes the aggregate discrepancy across observations.

\subsection{The Optimization Framework}

We now formulate the problem as minimizing the divergence between the observed plans and the corresponding model predictions. Specifically, we first seek a cost matrix $C \in \mathcal{C}$ that minimizes the sum of Kullback-Leibler (KL) divergences across all $K$ observations. This leads to the following bilevel optimization problem:
\begin{equation}\label{eq:bilevel}
    \inf_{C \in \mathcal{C}} \sum_{l=1}^K \text{KL}(\hat{X}^l \,||\, X^*(C, \mu^l, \nu^l)),
\end{equation}
where $X^*(C, \mu^l, \nu^l)$ is the unique solution to the forward entropy-regularized OT problem for the $l$-th observation.
Although conceptually straightforward,  \ref{eq:bilevel} is computationally challenging due to its nested structure. 
To derive a more tractable formulation, we follow \cite{10.1088/1361-6420/ae5ace,ma2020learning} and reformulate the M-IOT problem as a single-level, jointly convex optimization problem as follows:

\begin{equation}\label{eq:convex}
\begin{aligned}
    \inf_{C \in \mathcal{C}, \{u^l, v^l\}_{l=1}^K} J(u, v, C)
    &:= \gamma \sum_{l=1}^K \sum_{i,j=1}^n \exp\left(\frac{u_i^l + v_j^l - C_{ij}}{\gamma}\right) \\
    &\quad - \sum_{l=1}^K \langle \hat{X}^l, u^l \oplus v^l - C \rangle.
\end{aligned}
\end{equation}

The relationship between the bilevel and single-level formulations is formalized below.

\begin{proposition}
\label{prop:equivalence}
Assume that for each observation $l$, the observed transport plan $\hat{X}^l$ is strictly positive. Then, the bi-level optimization problem \ref{eq:bilevel} is equivalent to the single-level convex optimization  \ref{eq:convex}. The equivalence holds in the sense that if an optimal solution $C^*$ exists for one problem, an optimal solution exists for the other, and they coincide.
\end{proposition}
\begin{proof}
For a fixed cost matrix $C \in \mathcal{C}$ and each observation $l$, let $X^l_C = X^*(C, \mu^l, \nu^l)$ be the unique solution to the forward entropy-regularized OT problem \eqref{eq:forward_ot} and let $u^l,v^l$ to be the optimal dual variables corresponding to $X^l_C$. We first have
\begin{align*}
    \KL(\hat{X}^l \,||\, X^l_C) &= \sum_{i,j} (\hat{X}^l_{ij}\log \hat{X}^l_{ij} - \hat{X}^l_{ij}\log X^l_{C,ij}) - \sum_{i,j}\hat{X}^l_{ij} + \sum_{i,j}X^l_{C,ij} \\
    &= -H(\hat{X}^l) - \sum_{i,j} \hat{X}^l_{ij} \log X^l_{C,ij} + 1 \\
    &= -H(\hat{X}^l) - \frac{1}{\gamma}\sum_{i,j} \hat{X}^l_{ij} (u_i^l + v_j^l - C_{ij}) + 1,
\end{align*}
where the last equality follows from the optimality condition for the forward problem. Summing over all $K$ observations, the bilevel objective is:
\begin{equation}
\label{eq:proof_bilevel_obj}
    \sum_{l=1}^K \KL(\hat{X}^l \,||\, X^l_C) = K - \sum_l H(\hat{X}^l) - \frac{1}{\gamma} \sum_{l=1}^K \langle \hat{X}^l, u^l \oplus v^l - C \rangle.
\end{equation}

On the other hand:
\begin{align*}
    \min_{\{u^l,v^l\}} J(u,v,C) &= \gamma \sum_{l,i,j} \exp\left(\frac{u_{C,i}^l+v_{C,j}^l-C_{ij}}{\gamma}\right) - \sum_{l} \langle \hat{X}^l, u_C^l \oplus v_C^l - C \rangle \\
    &= \gamma \sum_{l,i,j} X^l_{C,ij} - \sum_l \langle \hat{X}^l, u_C^l \oplus v_C^l - C \rangle \\
    &= \gamma K - \sum_l \langle \hat{X}^l, u_C^l \oplus v_C^l - C \rangle.
\end{align*}
Minimizing this over $C \in \mathcal{C}$ is equivalent to minimizing $-\sum_l \langle \hat{X}^l, u_C^l \oplus v_C^l - C \rangle$, which is precisely  \eqref{eq:bilevel}. This establishes the equivalence.
\end{proof}

For notational convenience, let $z := (u^1, \dots, u^K, v^1, \dots, v^K, \text{vec}(C))^\top \in \mathbb{R}^{2nK + n^2}$. Denote the feasible set $\mathcal{Z} := (\mathbb{R}^n)^{2K} \times \mathcal{C}$, the optimal value by $J^*:=\inf_{z\in\mathcal Z}J(z)$, and the optimal solution set by $\mathcal Z^*:=\{z\in\mathcal Z:J(z)=J^*\}$. Let $A:\mathbb{R}^{2nK+n^2} \rightarrow \mathbb{R}^{Kn^2}$ be the linear operator such that the vector of exponential arguments is given by $y = Az$, where $y_{lij} = u_i^l + v_j^l - C_{ij}$. The objective function can be compactly written as $J(z) = \gamma \langle \exp(Az/\gamma), \mathbf{1}\rangle - \langle \text{vec}(\hat{X}), Az \rangle$, where $\text{vec}(\hat{X}) := (\text{vec}(\hat{X}^1),\ldots,\text{vec}(\hat{X}^K))^\top$.  

The gradient $\nabla J(z)$ and Hessian $\nabla^2 J(z)$ are given by:
\begin{align}
    \nabla J(z) &= A^\top \left( \exp\left(\frac{Az}{\gamma}\right) - \text{vec}(\hat{X}) \right), \label{eq:gradient_J}\\
    \nabla^2 J(z) &= \frac{1}{\gamma} A^\top D(z) A, \label{eq:hessian_J}
\end{align}
where $D(z) = \text{diag}(\exp(Az/\gamma))$ is a diagonal matrix with strictly positive entries. The structure of the Hessian implies that $\nabla^2 J(z)$ is positive semidefinite, ensuring that $J(z)$ is a convex function. The specific spectral properties of the operator $A$ will be analyzed in detail in Supplementary material for interested readers.

\section{Existence and Characterization of Solution Set}
\label{sec:analysis}

In this section, we study the existence and the structure of the optimal solution set of \eqref{eq:convex}. We first establish the existence of optimal solutions under mild regularity assumptions and then characterize the structure of the optimal solution set.

Before proceeding with the analysis, we introduce the following lemma.
\begin{lemma}\label{lem:g}
Assume $\text{vec}(\hat{X})$ is strictly positive. Define
\[
    g(y)=\gamma\langle \exp(y/\gamma),\mathbf{1}\rangle-\langle \text{vec}(\hat{X}),y\rangle,
    \qquad y\in\mathbb{R}^{Kn^2}.
\]
Then $g$ has a unique global minimizer and all its sublevel sets $\mathcal{S}_\alpha^g$ are compact.
\end{lemma}

\begin{proof}
Denote $b=\text{vec}(\hat{X})\in\mathbb{R}^{Kn^2}_{++}$. Since $\nabla g(y)=\exp(y/\gamma)-b$, the equation $\nabla g(y)=0$ has the unique solution $y^*=\gamma\log b$. Moreover, $\nabla^2 g(y)=\gamma^{-1}\diag\!\left(\exp(y/\gamma)\right)\succ 0$, so $g$ is strictly convex and $y^*$ is its unique global minimizer. Denote $g_q(t):=\gamma\exp(t/\gamma)-b_qt$, then $g(y)=\sum_{q}g_q(y_q)$. For every $q$, $g_q(t)\to+\infty$ as $t\to\pm\infty$. Let $m_q:=\min_{t\in\mathbb{R}}g_q(t)$. If $y\in\mathcal{S}_\alpha^g$, then for every $q$,
\[
    g_q(y_q)
    =g(y)-\sum_{p\neq q}g_p(y_p)
    \leq \alpha-\sum_{p\neq q}m_p.
\]
Therefore $\mathcal{S}_\alpha^g$ is bounded and hence compact.
\end{proof}

We now impose the following regularity assumption on the feasible set $\mathcal{Z}$.

\begin{asmp}
\label{assump:closed_projection}
Either $\mathcal{Z}$ is a polyhedral set, or $R_\mathcal{Z}\cap\ker(A)\subseteq L_\mathcal{Z}$, where $R_\mathcal{Z}$ and $L_\mathcal{Z}$ denote the recession cone and lineality space of $\mathcal{Z}$, respectively.
\end{asmp}

\begin{remark}
For every nonempty sublevel set
\[
    \mathcal{S}_\alpha^J:=\{z\in\mathcal{Z}:J(z)\leq\alpha\}
    =\{z\in\mathcal{Z}:Az\in\mathcal{S}_\alpha^g\},
\]
Lemma~\ref{lem:g} and \cite[Proposition 1.4.2(d)]{bertsekas2009convex} imply $R_{\mathcal{S}_\alpha^J}=R_\mathcal{Z}\cap\ker(A)$.
Thus, the second condition in Assumption~\ref{assump:closed_projection} means that every recession direction of an objective sublevel set that is invisible to $J$ is in fact a lineality direction of the feasible set. In particular, if $R_\mathcal{Z}\cap\ker(A)=\{0\}$, then every nonempty sublevel set of $J$ is bounded and Assumption~\ref{assump:closed_projection} is automatically satisfied.
\end{remark}

The following observation will also be useful in the convergence analysis. It formalizes the fact that the sublevel sets are compact after quotienting out their lineality directions.

\begin{lemma}\label{lem:compact_mod_lineality}
Assume $\mathcal{N}:=R_\mathcal{Z}\cap\ker(A)\subseteq L_\mathcal{Z}$.
Then $\mathcal{N}$ is a linear subspace. For every nonempty sublevel set $\mathcal{S}_\alpha^J$, with $\mathcal{W}:=\mathcal{N}^\perp$, one has $ \mathcal{S}_\alpha^J
=\big(\mathcal{S}_\alpha^J\cap\mathcal{W}\big)+\mathcal{N}$ and the canonical slice $\mathcal{S}_\alpha^J\cap\mathcal{W}$ is compact.
\end{lemma}

\begin{proof}
Since $\mathcal{N}\subseteq L_\mathcal{Z}$, for every $d\in\mathcal{N}$ one also has $-d\in R_\mathcal{Z}$. Because $d\in\ker(A)$, this shows that $-d\in\mathcal{N}$, and hence $\mathcal{N}$ is a linear subspace. Moreover, $J(z+d)=J(z)$ for every $d\in\mathcal{N}$, while $\mathcal{Z}+d=\mathcal{Z}$ because $d\in L_\mathcal{Z}$. Hence $\mathcal{N}\subseteq L_{\mathcal{S}_\alpha^J}$. Together with $R_{\mathcal{S}_\alpha^J}=\mathcal{N}$, this yields $L_{\mathcal{S}_\alpha^J}=R_{\mathcal{S}_\alpha^J}=\mathcal{N}$.

For any $z\in\mathcal{S}_\alpha^J$, decompose $z=w+d$ with $w=\Pi_\mathcal{W}z$ and $d=\Pi_\mathcal{N}z\in\mathcal{N}$. Since $-d\in L_{\mathcal{S}_\alpha^J}$, we have $w=z-d\in\mathcal{S}_\alpha^J$. Therefore $ \mathcal{S}_\alpha^J
=\big(\mathcal{S}_\alpha^J\cap\mathcal{W}\big)+\mathcal{N}$. The set $\mathcal{S}_\alpha^J\cap\mathcal{W}$ is closed and convex, and its recession cone is
\[
    R_{\mathcal{S}_\alpha^J\cap\mathcal{W}}
    =R_{\mathcal{S}_\alpha^J}\cap\mathcal{W}
    =\mathcal{N}\cap\mathcal{N}^\perp
    =\{0\}.
\]
Hence it is bounded, and therefore compact in finite dimensions.
\end{proof}

\begin{proposition}
\label{prop:existence}
Under Assumption~\ref{assump:closed_projection}, the M-IOT optimization problem \eqref{eq:convex} admits at least one optimal solution.
\end{proposition}

\begin{proof}
The objective $J(z)=g(Az)$ is invariant along $\ker(A)$. Let $\mathcal{W}_A:=\ker(A)^\perp$ and let $\Pi_{\mathcal{W}_A}$ denote the orthogonal projection onto $\mathcal{W}_A$. Define
\[
    \widetilde J:\Pi_{\mathcal{W}_A}(\mathcal{Z})\to\mathbb{R},
    \qquad
    \widetilde J(w):=J(z)
\]
for any $z\in\mathcal{Z}$ satisfying $\Pi_{\mathcal{W}_A}z=w$. This is well-defined because two such preimages differ by an element of $\ker(A)$.

By Assumption~\ref{assump:closed_projection} and \cite[Proposition 1.4.13]{bertsekas2009convex}, the projected feasible set $\Pi_{\mathcal{W}_A}(\mathcal{Z})$ is closed. It is also convex. Consider a nonempty sublevel set
\[
    \widetilde{\mathcal{S}}_\alpha
    :=\{w\in\Pi_{\mathcal{W}_A}(\mathcal{Z}):\widetilde J(w)\leq\alpha\} =\{w\in\Pi_{\mathcal{W}_A}(\mathcal{Z}):Aw\in\mathcal{S}_\alpha^g\}.
\]
Lemma~\ref{lem:g} and \cite[Proposition 1.4.2(d)]{bertsekas2009convex} give $R_{\widetilde{\mathcal{S}}_\alpha}
=R_{\Pi_{\mathcal{W}_A}(\mathcal{Z})}\cap\ker(A)$. But $\Pi_{\mathcal{W}_A}(\mathcal{Z})\subseteq\mathcal{W}_A$ and $\mathcal{W}_A\cap\ker(A)=\{0\}$, so
$R_{\widetilde{\mathcal{S}}_\alpha}=\{0\}$. Hence every nonempty sublevel set of $\widetilde J$ is compact.

Since $\widetilde J$ is continuous and bounded below, choose any $\alpha>\inf\widetilde J$ for which $\widetilde{\mathcal{S}}_\alpha$ is nonempty. The minimum of $\widetilde J$ over this compact sublevel set is also its global minimum. Lifting a minimizer back to $\mathcal{Z}$ yields an optimal solution of \eqref{eq:convex}.
\end{proof}

We next characterize the entire optimal solution set.

\begin{proposition}
\label{prop:optimal_set_structure}
Let $\mathcal{Z}=(\R^n)^{2K}\times\mathcal{C}$, where $\mathcal{C}$ is closed and convex, and assume that the optimal solution set $\mathcal{Z}^*$ is nonempty. Let $z^*\in\mathcal{Z}^*$ be arbitrary. Then
\begin{equation}
    \mathcal{Z}^*=(z^*+\ker(A))\cap\mathcal{Z}.
\end{equation}
In particular, all optimal solutions have the same image under $A$.
\end{proposition}

\begin{proof}
Let $z_1,z_2\in\mathcal{Z}^*$ be two optimal solutions. If $Az_1\neq Az_2$, then the convexity of $\mathcal{Z}$ and the strict convexity of $g$ imply
\[
    J\!\left(\frac{z_1+z_2}{2}\right)
    =g\!\left(\frac{Az_1+Az_2}{2}\right)
    <\frac{g(Az_1)+g(Az_2)}{2}
    =J^*,
\]
a contradiction. Thus $Az_1=Az_2$ for all $z_1,z_2\in\mathcal{Z}^*$. In particular, for any $z\in\mathcal{Z}^*$, $A(z-z^*)=0$, so $z\in(z^*+\ker(A))\cap\mathcal{Z}$.

Conversely, if $\tilde z\in(z^*+\ker(A))\cap\mathcal{Z}$, then $A\tilde z=Az^*$ and therefore $J(\tilde z)=g(A\tilde z)=g(Az^*)=J^*$. Hence $\tilde z\in\mathcal{Z}^*$, proving the claimed characterization.
\end{proof}

\section{A Block Coordinate Descent with Anderson-type Extrapolation (BCDwAE) Algorithm}
\label{sec:PRAABCD}

In this section, we propose an efficient algorithm to solve \eqref{eq:convex}. Motivated by Anderson acceleration \cite{anderson1965iterative}, we propose a Block Coordinate Descent with Anderson-type Extrapolation (BCDwAE) algorithm.

\subsection{Algorithm Derivation}

One iteration of the base BCD scheme produces an intermediate point $z^{k+}=(u^{k+},v^{k+},C^{k+})$ from the current iterate $z^k=(u^k,v^k,C^k)$. The $u$- and $v$-blocks are minimized exactly and sequentially, while the $C$-block is updated by one projected-gradient step. For fixed $C^k$ and $v^k$, the update for $u$ is obtained from $\nabla_uJ(u,v^k,C^k)=0$. Subsequently, $v$ is updated using the new $u$. For each observation $l\in\{1,\dots,K\}$,
\begin{align}
    u^{l,k+}_i
    &=\gamma\log\mu_i^l
    -\gamma\log\left(\sum_{j=1}^n\exp\left(\frac{v_j^{l,k}-C_{ij}^k}{\gamma}\right)\right),
    \quad \forall i, \label{eq:update_u}\\
    v^{l,k+}_j
    &=\gamma\log\nu_j^l
    -\gamma\log\left(\sum_{i=1}^n\exp\left(\frac{u_i^{l,k+}-C_{ij}^k}{\gamma}\right)\right),
    \quad \forall j. \label{eq:update_v}
\end{align}
The update for $C$ enforces $C\in\mathcal{C}$ via a projected-gradient step. Let $\tilde z^k=(u^{k+},v^{k+},C^k)$. The gradient with respect to $C$ at this intermediate point is
\begin{equation}
    \nabla_CJ(\tilde z^k)
    =\sum_{l=1}^K\hat X^l
    -\sum_{l=1}^K\exp\left(\frac{u^{l,k+}\oplus v^{l,k+}-C^k}{\gamma}\right).
\end{equation}
For a fixed step size $\eta>0$,
\begin{equation}
    C^{k+}=\Proj_\mathcal{C}\left(C^k-\eta\nabla_CJ(\tilde z^k)\right). \label{eq:update_C}
\end{equation}
We call $z^{k+}$ the base BCD point associated with $z^k$ and define the corresponding residual by
\[
    r^k:=z^{k+}-z^k.
\]
Standard BCD iterations may converge slowly. We therefore apply a regularized damped Anderson-type extrapolation to these residuals, following \cite{wei2021stochastic}.
Let $m$ be the memory depth and $m_k=\min(k,m)$. Define
\begin{equation}
    R^k=[\Delta r^{k-m_k},\dots,\Delta r^{k-1}],
    \qquad
    Z^k=[\Delta z^{k-m_k},\dots,\Delta z^{k-1}],
\end{equation}
where $\Delta r^i=r^{i+1}-r^i$ and $\Delta z^i=z^{i+1}-z^i$. The coefficient vector is obtained from
\begin{equation}
    \Gamma_k^*
    =\arg\min_\Gamma
    \left\|r^k-R^k\Gamma\right\|_2^2
    +\delta_k\left\|Z^k\Gamma\right\|_2^2,
\end{equation}
where $\delta_k>0$. Let $\beta_k\in(0,1]$ be the mixing factor and $\alpha_k\in[0,1]$ the damping parameter. The accelerated iterate is
\begin{equation}
    z^{k+1}
    =\Proj_\mathcal{Z}\left(
        z^k+\beta_kr^k
        -\alpha_k(Z^k+\beta_kR^k)\Gamma_k^*
    \right). \label{eq:anderson_update}
\end{equation}
The projection acts as $\Proj_\mathcal{C}$ on the $C$-block and as the identity on the $u$- and $v$-blocks.

\begin{algorithm}[t]
\caption{Block Coordinate Descent with Anderson-type Extrapolation (BCDwAE) for M-IOT}\label{alg:PRAABCD}
\begin{algorithmic}[1]
\REQUIRE Observations $\{\hat X^l,\mu^l,\nu^l\}_{l=1}^K$, parameters $\gamma,\eta,\delta_k,\beta_k,\alpha_k,m$, tolerance $\epsilon$.
\STATE Initialize $z^0=(u^0,v^0,C^0)\in\mathcal{Z}$, $R_{hist}= []$, $Z_{hist}= []$.
\FOR{$k=0,1,2,\dots$}
    \STATE Update $u^{k+}$ via \eqref{eq:update_u}.
    \STATE Update $v^{k+}$ via \eqref{eq:update_v}.
    \STATE Update $C^{k+}=\Proj_\mathcal{C}(C^k-\eta\nabla_CJ(u^{k+},v^{k+},C^k))$.
    \STATE Set $z^{k+}\leftarrow(u^{k+},v^{k+},C^{k+})$ and $r^k\leftarrow z^{k+}-z^k$.
    \IF{$\|r^k\|<\epsilon$} \STATE \textbf{break} \ENDIF
    \IF{$k>0$}
        \STATE $\Delta r^{k-1}\leftarrow r^k-r^{k-1}$, $\Delta z^{k-1}\leftarrow z^k-z^{k-1}$.
        \STATE Update $R^k,Z^k$ with the newest columns while maintaining memory depth $m$.
    \ENDIF
    \IF{$k=0$ or history is empty}
        \STATE $z^{k+1}\leftarrow\Proj_\mathcal{Z}(z^k+\beta_kr^k)$.
    \ELSE
        \STATE $\Gamma_k^*\leftarrow\arg\min_\Gamma\|r^k-R^k\Gamma\|_2^2+\delta_k\|Z^k\Gamma\|_2^2$.
        \STATE $z^{k+1}\leftarrow\Proj_\mathcal{Z}(z^k+\beta_kr^k-\alpha_k(Z^k+\beta_kR^k)\Gamma_k^*)$.
    \ENDIF
\ENDFOR
\RETURN $C^k$.
\end{algorithmic}
\end{algorithm}

\subsection{Convergence Analysis}

We now establish the linear convergence of the unsafeguarded BCDwAE iteration. We make the following assumption throughout this section.

\begin{asmp}\label{asmp:convergence}
\begin{itemize}
    \item[(1)] All observed transport plans satisfy $\hat X^l\in\mathbb{R}_{++}^{n\times n}$ for $1\leq l\leq K$.
    \item[(2)] $\mathcal{C}$ is nonempty, closed, and convex, and $\mathcal{N}:=R_\mathcal{Z}\cap\ker(A)\subseteq L_\mathcal{Z}$.
    \item[(3)] There exists $z^*\in\mathcal{Z}^*$ such that $\operatorname{ri}(\mathcal{Z})\cap(z^*+\ker(A))\neq\emptyset$.
\end{itemize}
\end{asmp}

\begin{remark}
Assumption~\ref{asmp:convergence}(2) is the same structural condition used in Section~\ref{sec:analysis} to control recession directions invisible to the objective, and it implies existence by Proposition~\ref{prop:existence}. If $\mathcal{C}$ is affine, then $R_\mathcal{Z}=L_\mathcal{Z}$, so Assumption~\ref{asmp:convergence}(2) is automatic; moreover $\operatorname{ri}(\mathcal{Z})=\mathcal{Z}$, and hence Assumption~\ref{asmp:convergence}(3) is automatic once $\mathcal{Z}^*\neq\emptyset$. The relative-interior condition also allows non-affine constraints, for example inequality constraints, provided the optimal affine solution set intersects the relative interior of the feasible set.
\end{remark}

Let $r^k=z^{k+}-z^k$ denote the residual between the current iterate and its base BCD point, and define $F(z):=J(u,v,C)+\mathbf{1}_\mathcal{C}(C)$ and $F^*:=\min F=J^*$, where $\mathbf{1}_\mathcal{C}(C)=0$ for $C\in\mathcal C$ and $+\infty$ otherwise. Here $\partial F$ denotes the convex subdifferential of $F$. For the $k$-th base iteration, let $z_u^k=(u^{k+},v^k,C^k)$ and $z_v^k=(u^{k+},v^{k+},C^k)$.
Then $\nabla_uJ(z_u^k)=0$ and $\nabla_vJ(z_v^k)=0$, while $C^{k+}$ is obtained by one projected-gradient step from $z_v^k$.

Let $M^kr^k:=\alpha_k(Z^k+\beta_kR^k)\Gamma_k^*$ and $H^kr^k:=\beta_kr^k-M^kr^k$, so that $z^{k+1}=\Proj_\mathcal{Z}(z^k+H^kr^k)$.

\begin{lemma}\label{lem:Anderson_bound}
Suppose $0\leq\alpha_k\leq1$, $\beta_k>0$, $\delta_k>0$, and $\delta_k^{-1}\leq c_2\beta_k^2$ for some constant $c_2>0$. Then
\begin{align}
    \|H^kr^k\|^2
    &\leq 2(1+c_2)\beta_k^2\|r^k\|^2, \label{eq:H_bound}\\
    \|M^kr^k\|
    &\leq \alpha_k\beta_k(1+\sqrt{c_2})\|r^k\|. \label{eq:M_bound}
\end{align}
\end{lemma}

The bound \eqref{eq:H_bound} follows from \cite[Lemma~1 in the supplementary material]{wei2021stochastic} under the stated parameter restrictions. A self-contained proof of both \eqref{eq:H_bound} and the companion estimate \eqref{eq:M_bound} is given in Section~\ref{app:anderson_bound} of the supplementary material.

We next show that all constants needed below can be chosen on a fixed set that is compact modulo $\mathcal{N}$.

\begin{lemma}\label{lem:localization}
Suppose Assumption~\ref{asmp:convergence} holds and $\delta_k^{-1}\leq c_2\beta_k^2$ with $0<\beta_k\leq1$ and $0\leq\alpha_k\leq1$. Fix an arbitrary reference step-size bound $\bar\eta>0$ and restrict the projected-gradient step to $0<\eta\leq\bar\eta$. Let
\[
    \ell_0:=J(z^0),
    \qquad
    \mathcal{S}_0:=\{z\in\mathcal{Z}:J(z)\leq\ell_0\},
    \qquad
    \mathcal{W}:=\mathcal{N}^\perp.
\]
Then there exists a convex set $\mathcal{B}\subseteq\mathcal{Z}$, depending on $\bar\eta$ but not on the actual choice of $\eta\in(0,\bar\eta]$, such that
\[
    \mathcal{B}+\mathcal{N}=\mathcal{B},
    \qquad
    \mathcal{B}\cap\mathcal{W}\ \text{is compact}.
\]
Moreover, if $z^0,\ldots,z^k\in\mathcal{S}_0$, then every point
\[
    z^j,\ z_u^j,\ z_v^j,\ z^{j+},\ z^{j+1},
    \qquad 0\leq j\leq k,
\]
that is defined at those iterations belongs to $\mathcal{B}$. Consequently, there exist constants $L_{\bar\eta}>0$ and $\kappa_{\bar\eta}>0$, independent of $\eta\in(0,\bar\eta]$, such that $\nabla J$ is $L_{\bar\eta}$-Lipschitz on $\mathcal{B}$ and the $u$- and $v$-block subproblems are uniformly $\kappa_{\bar\eta}$-strongly convex on $\mathcal{B}$.
\end{lemma}

\begin{proof}
See Section~\ref{app:localization} of the supplementary material.
\end{proof}

Fix $\bar\eta>0$ and let $L:=L_{\bar\eta}$ and $\kappa:=\kappa_{\bar\eta}$ be the constants from Lemma~\ref{lem:localization}. In what follows, the projected-gradient step size is chosen so that
\begin{equation}\label{eq:eta_choice}
    0<\eta<\min\left\{\bar\eta,\frac{2}{L}\right\}.
\end{equation}

\begin{lemma}
\label{lem:sufficient_decrease}
If $z^k\in\mathcal{S}_0$, then
\[
    J(z^k)-J(z^{k+})
    \geq c_1\|z^k-z^{k+}\|^2
    =c_1\|r^k\|^2,
\]
where $c_1:=\min\left\{\kappa/2,1/\eta-L/2\right\}>0$.
\end{lemma}

\begin{proof}
By exact minimization of the $u$- and $v$-blocks and their uniform strong convexity,
\[
J(z^k)-J(z_u^k) \geq\frac{\kappa}{2}\|u^k-u^{k+}\|^2,\text{ and }
J(z_u^k)-J(z_v^k) \geq\frac{\kappa}{2}\|v^k-v^{k+}\|^2.
\]
For the $C$-block, the descent lemma gives
\[
    J(z^{k+})
    \leq J(z_v^k)
    +\langle\nabla_CJ(z_v^k),C^{k+}-C^k\rangle
    +\frac{L}{2}\|C^{k+}-C^k\|^2.
\]
The projection optimality condition for \eqref{eq:update_C}, evaluated at $C=C^k$, yields
\[
    \langle\nabla_CJ(z_v^k),C^{k+}-C^k\rangle
    \leq-\frac{1}{\eta}\|C^{k+}-C^k\|^2.
\]
Therefore
\[
    J(z_v^k)-J(z^{k+})
    \geq\left(\frac{1}{\eta}-\frac{L}{2}\right)
      \|C^{k+}-C^k\|^2.
\]
Summing the three block decreases proves the result.
\end{proof}

\begin{lemma}
\label{lem:error_bound}
If $z^k\in\mathcal{S}_0$, then
\[
    \|r^k\|^2
    \geq w\,\operatorname{dist}(0,\partial F(z^{k+}))^2,
    \qquad
    w:=\left(3L+\frac{1}{\eta}\right)^{-2}>0.
\]
\end{lemma}

\begin{proof}
The optimality condition for the $C$-update implies that there exists $\xi^k\in N_\mathcal{C}(C^{k+})$ such that $\xi^k+\nabla_CJ(z_v^k)+\eta^{-1}(C^{k+}-C^k)=0$.
Thus
\[
    g^{k+}:=
    \big(\nabla_uJ(z^{k+}),\nabla_vJ(z^{k+}),
    \nabla_CJ(z^{k+})+\xi^k\big)
    \in\partial F(z^{k+}).
\]
Using $\nabla_uJ(z_u^k)=0$, $\nabla_vJ(z_v^k)=0$, and the $L$-Lipschitz continuity of $\nabla J$ on the localized set,
\begin{align*}
    \|\nabla_uJ(z^{k+})\|
    &\leq L\|z^{k+}-z_u^k\|
    \leq L\|r^k\|,\\
    \|\nabla_vJ(z^{k+})\|
    &\leq L\|z^{k+}-z_v^k\|
    \leq L\|r^k\|,\\
    \|\nabla_CJ(z^{k+})+\xi^k\|
    &\leq\left(L+\frac{1}{\eta}\right)\|C^{k+}-C^k\|
    \leq\left(L+\frac{1}{\eta}\right)\|r^k\|.
\end{align*}
Consequently,
\[
    \operatorname{dist}(0,\partial F(z^{k+}))
    \leq\|g^{k+}\|
    \leq\left(3L+\frac{1}{\eta}\right)\|r^k\|,
\]
which proves the claim.
\end{proof}

We next establish quadratic growth. The statement below is written on sets that are compact modulo the lineality subspace, which is the form needed for the present problem.

\begin{theorem}
\label{thm:qgc}
Let $\mathcal{Z}$ be closed and convex and suppose $\mathcal{Z}^*\neq\emptyset$. Let $\mathcal{N}\subseteq L_\mathcal{Z}\cap\ker(A)$ be a linear subspace, let $\mathcal{W}=\mathcal{N}^\perp$, and let $\mathcal{B}\subseteq\mathcal{Z}$ be convex and satisfy $\mathcal{B}+\mathcal{N}=\mathcal{B}$ with $\mathcal{B}\cap\mathcal{W}$ compact. Choose any $z^*\in\mathcal{Z}^*$. Assume either that $\mathcal{Z}$ is polyhedral or that $\operatorname{ri}(\mathcal{Z})\cap(z^*+\ker(A))\neq\emptyset$.
Then there exists $\sigma>0$ such that, for all $z\in\mathcal{B}$,
\begin{equation}
    J(z)-J^*
    \geq\frac{\sigma}{2}\operatorname{dist}(z,\mathcal{Z}^*)^2.
    \label{eq:qgc}
\end{equation}
\end{theorem}

\begin{proof}
Because $\mathcal{N}\subseteq L_\mathcal{Z}\cap\ker(A)$, both $J$ and the sets $\mathcal{Z}$ and $\mathcal{Z}^*$ are invariant under translations by $\mathcal{N}$. Hence
\[
    J(z)=J(\Pi_\mathcal{W}z),
    \qquad
    \operatorname{dist}(z,\mathcal{Z}^*)
    =\operatorname{dist}(\Pi_\mathcal{W}z,\mathcal{Z}^*).
\]
It therefore suffices to prove the estimate for $z\in\mathcal{B}\cap\mathcal{W}$, which is compact.

Let $z_p=\Proj_{\mathcal{Z}^*}(z)$ and $v=z-z_p$. By optimality of $z_p$ over the convex feasible set, $\langle\nabla J(z_p),v\rangle\geq0$.
Applying Taylor's theorem to the scalar function $t\mapsto J(z_p+tv)$ gives, for some $\tilde z$ on the segment $[z_p,z]$,
\[
    J(z)-J^*
    \geq\frac12v^T\nabla^2J(\tilde z)v
    =\frac{1}{2\gamma}\|D(\tilde z)^{1/2}Av\|^2.
\]
By Proposition~\ref{prop:optimal_set_structure}, all optimal solutions have the same image $Az_p$. Since $A(\mathcal{B}\cap\mathcal{W})$ is compact, the vectors $A\tilde z$, which lie on segments joining this compact set to the fixed optimal image, remain in a compact set. Hence there exists $D_{\min}>0$ such that
\[
    J(z)-J^*
    \geq\frac{D_{\min}}{2\gamma}\|Av\|^2.
\]

Let $S_{aff}:=z^*+\ker(A)$. Proposition~\ref{prop:optimal_set_structure} gives $\mathcal{Z}^*=S_{aff}\cap\mathcal{Z}$.
Under either assumption of the theorem, the pair $\{S_{aff},\mathcal{Z}\}$ is boundedly linearly regular \cite{bauschke1996projection,bauschke1999strong}. Therefore, on the compact set $\mathcal{B}\cap\mathcal{W}$ there exists $\kappa_{\rm BLR}>0$ such that, since $z\in\mathcal{Z}$,
\[
    \|v\|
    =\operatorname{dist}(z,\mathcal{Z}^*)
    \leq\kappa_{\rm BLR}\operatorname{dist}(z,S_{aff}).
\]
Since $S_{aff}=\{y:Ay=Az_p\}$,
\[
    \operatorname{dist}(z,S_{aff})
    =\|A^\dagger(Az-Az_p)\|
    \leq\frac{1}{\sigma_{\min}^+(A)}\|Av\|.
\]
Consequently,
\[
    \|Av\|
    \geq\frac{\sigma_{\min}^+(A)}{\kappa_{\rm BLR}}\|v\|.
\]
Combining the preceding estimates proves \eqref{eq:qgc} with
\[
    \sigma
    =\frac{D_{\min}(\sigma_{\min}^+(A))^2}
    {\gamma\kappa_{\rm BLR}^2}>0.
\]
\end{proof}

The Polyak--{\L}ojasiewicz (PL) condition \cite{polyak1963gradient} admits a subgradient formulation for nonsmooth functions \cite[Section~3]{liao2024error}. For the proper closed convex function $F$, the standard quadratic-growth-to-PL implication \cite[Theorem~3.1]{liao2024error} applies pointwise on $\mathcal{B}$ and yields the following result.

\begin{lemma}[Subgradient PL condition]\label{lem:qgc_pl}
Suppose \eqref{eq:qgc} holds on $\mathcal{B}$. Then, for every $z\in\mathcal{B}$,
\begin{equation}
    F(z)-F^*
    \leq\frac{2}{\sigma}\operatorname{dist}(0,\partial F(z))^2.
    \label{eq:qgc_subgrad}
\end{equation}
\end{lemma}

The relative-interior assumption provides the missing control of the smooth gradient along feasible directions.

\begin{lemma}\label{lem:tangential_gradient}
Let $\mathcal{S}:=\operatorname{par}(\mathcal{Z})=\operatorname{span}(\mathcal{Z}-\mathcal{Z})$. Suppose Assumption~\ref{asmp:convergence} holds. Then $\Pi_\mathcal{S}\nabla J(z^*)=0$ for all $z^*\in\mathcal{Z}^*$.
Moreover, for every $z\in\mathcal{B}$,
\begin{equation}
    \|\Pi_\mathcal{S}\nabla J(z)\|
    \leq L\operatorname{dist}(z,\mathcal{Z}^*),
    \label{eq:tangential_gradient_bound}
\end{equation}
where $L$ may be chosen as the Lipschitz constant in Lemma~\ref{lem:localization} after enlarging it if necessary.
\end{lemma}

\begin{proof}
By Assumption~\ref{asmp:convergence}(3), there exists $\bar z\in\operatorname{ri}(\mathcal{Z})\cap(z^*+\ker(A))$.
Proposition~\ref{prop:optimal_set_structure} implies $\bar z\in\mathcal{Z}^*$. Since $\bar z$ is an optimal solution, $-\nabla J(\bar z)\in N_\mathcal{Z}(\bar z)$. At a relative-interior point of a closed convex set, $N_\mathcal{Z}(\bar z)=\mathcal{S}^\perp$.
Thus $\Pi_\mathcal{S}\nabla J(\bar z)=0$. By Proposition~\ref{prop:optimal_set_structure}, all optimal solutions have the same image under $A$, and hence the same gradient by \eqref{eq:gradient_J}. Therefore $\Pi_\mathcal{S}\nabla J(z^*)=0$ for every $z^*\in\mathcal{Z}^*$.

For $z\in\mathcal{B}$, let $z_p=\Proj_{\mathcal{Z}^*}(z)$. Since $J$ and $\nabla J$ are invariant along $\mathcal{N}$, we may replace $z$ and $z_p$ by their canonical representatives in $\mathcal{W}$. The canonical representative of $z_p$ belongs to the initial sublevel slice $\widehat{\mathcal{S}}_0$, while that of $z$ belongs to the compact localized set. Hence the Lipschitz estimate applies to the segment joining them, and
\[
    \|\Pi_\mathcal{S}\nabla J(z)\|
    =\|\Pi_\mathcal{S}(\nabla J(z)-\nabla J(z_p))\|
    \leq L\|z-z_p\|.
\]
This is \eqref{eq:tangential_gradient_bound}.
\end{proof}

We can now establish the main convergence theorem. For simplicity, assume $\alpha_k=\alpha$, $\beta_k=\beta$, and $\delta_k=\delta$ for all $k$. Define $K_1 :=L(1+\sqrt{c_2})\left(1+\frac{2}{\sigma\sqrt{w}}\right)$, $K_2 :=L(1+c_2)$, and $\mu :=\frac{2L}{\sigma\sqrt{w}}+\frac{L}{2}+\frac{2}{\sigma w}$. 

\begin{theorem}
\label{thm:linear_convergence}
Suppose Assumption~\ref{asmp:convergence} holds. Fix $\bar\eta>0$, let $L=L_{\bar\eta}$ and $\kappa=\kappa_{\bar\eta}$ be as above, and choose $\eta$ according to \eqref{eq:eta_choice}. Assume $0<\beta\leq1$, $0<\alpha<1$, and $\delta^{-1}\leq c_2\beta^2$. Let $c_1$ be the constant in Lemma~\ref{lem:sufficient_decrease}, $w$ the constant in Lemma~\ref{lem:error_bound}, and $\sigma$ the quadratic-growth constant in Theorem~\ref{thm:qgc} on the set $\mathcal{B}$. If
\begin{equation}\label{eq:conv_condition}
    0<\alpha<\frac{c_1}{K_1},
    \qquad
    0<\beta<\frac{c_1-\alpha K_1}{K_2},
    \qquad
    \lambda:=\beta(c_1-\alpha K_1-\beta K_2)<\mu,
\end{equation}
then $\{F(z^k)\}$ converges Q-linearly to $F^*$, and there exists $z^\infty\in\mathcal{Z}^*$ such that $z^k\to z^\infty$ R-linearly.
\end{theorem}

\begin{proof}
Let $V_k:=F(z^k)-F^*=J(z^k)-F^*$.
We first prove by induction that $z^k\in\mathcal{S}_0$ for all $k$. The claim holds for $k=0$. Suppose it holds up to iteration $k$. Lemma~\ref{lem:localization} then guarantees that all points used below lie in the fixed localized set $\mathcal{B}$, so the constants $L,\kappa,\sigma$ are valid.

\textbf{Step 1: Sufficient decrease of the accelerated iterate.}
Let $p_k:=z^{k+1}-z^k$ and $y^k:=z^k+\beta r^k=(1-\beta)z^k+\beta z^{k+}$.
Since $\mathcal{Z}$ is convex and $0<\beta\leq1$, $y^k\in\mathcal{Z}$. By the descent lemma,
\begin{equation}
    J(z^{k+1})-J(z^k)
    \leq\langle\nabla J(z^k),p_k\rangle
    +\frac{L}{2}\|p_k\|^2.
    \label{eq:descent_lemma_aa}
\end{equation}
Decompose $p_k=\beta r^k+(z^{k+1}-y^k)$.
By Lemma~\ref{lem:sufficient_decrease} and convexity of $J$,
\[
    \langle\nabla J(z^k),r^k\rangle
    \leq J(z^{k+})-J(z^k)
    \leq-c_1\|r^k\|^2,
\]
and hence
\begin{equation}
    \langle\nabla J(z^k),\beta r^k\rangle
    \leq-\beta c_1\|r^k\|^2.
    \label{eq:base_direction_decrease}
\end{equation}

We next control the Anderson perturbation. Since $z^{k+1},y^k\in\mathcal{Z}$, $z^{k+1}-y^k\in\mathcal{S}:=\operatorname{par}(\mathcal{Z})$.
Thus
\[
    \langle\nabla J(z^k),z^{k+1}-y^k\rangle
    =\langle\Pi_\mathcal{S}\nabla J(z^k),z^{k+1}-y^k\rangle.
\]
Lemma~\ref{lem:error_bound} and Lemma~\ref{lem:qgc_pl} give
\begin{equation}
    F(z^{k+})-F^*
    \leq\frac{2}{\sigma w}\|r^k\|^2.
    \label{eq:gap_at_base_point}
\end{equation}
Combining this with quadratic growth yields
\[
    \operatorname{dist}(z^{k+},\mathcal{Z}^*)
    \leq\frac{2}{\sigma\sqrt{w}}\|r^k\|,
\]
and therefore
\begin{equation}
    \operatorname{dist}(z^k,\mathcal{Z}^*)
    \leq\left(1+\frac{2}{\sigma\sqrt{w}}\right)\|r^k\|.
    \label{eq:dist_zk_residual}
\end{equation}
By Lemma~\ref{lem:tangential_gradient},
\[
    \|\Pi_\mathcal{S}\nabla J(z^k)\|
    \leq L\left(1+\frac{2}{\sigma\sqrt{w}}\right)\|r^k\|.
\]
Moreover, projection nonexpansiveness and \eqref{eq:M_bound} imply
\[
    \|z^{k+1}-y^k\|
    \leq\|M^kr^k\|
    \leq\alpha\beta(1+\sqrt{c_2})\|r^k\|.
\]
Consequently,
\begin{equation}
    \left|\langle\nabla J(z^k),z^{k+1}-y^k\rangle\right|
    \leq\alpha\beta K_1\|r^k\|^2.
    \label{eq:AA_inner_bound}
\end{equation}
Finally, by projection nonexpansiveness and \eqref{eq:H_bound},
\[
    \|p_k\|^2
    \leq\|H^kr^k\|^2
    \leq2(1+c_2)\beta^2\|r^k\|^2,
\]
so
\[
    \frac{L}{2}\|p_k\|^2
    \leq\beta^2K_2\|r^k\|^2.
\]
Substituting these estimates into \eqref{eq:descent_lemma_aa} gives
\begin{equation}
    V_{k+1}-V_k
    \leq-\lambda\|r^k\|^2,
    \qquad
    \lambda:=\beta(c_1-\alpha K_1-\beta K_2)>0.
    \label{eq:step1_decrease}
\end{equation}
In particular, $J(z^{k+1})\leq J(z^k)\leq\ell_0$, so $z^{k+1}\in\mathcal{S}_0$. This closes the induction and shows that all iterates remain in the initial sublevel set without assuming their boundedness.

\textbf{Step 2: Relating the potential gap to the residual.}
Using smoothness at $z^{k+}$,
\begin{align*}
    J(z^k)-J(z^{k+})
    &\leq\langle\nabla J(z^{k+}),z^k-z^{k+}\rangle
      +\frac{L}{2}\|r^k\|^2\\
    &=\langle\Pi_\mathcal{S}\nabla J(z^{k+}),z^k-z^{k+}\rangle
      +\frac{L}{2}\|r^k\|^2.
\end{align*}
By \eqref{eq:gap_at_base_point}, quadratic growth, and Lemma~\ref{lem:tangential_gradient},
\[
    \|\Pi_\mathcal{S}\nabla J(z^{k+})\|
    \leq\frac{2L}{\sigma\sqrt{w}}\|r^k\|.
\]
Hence
\[
    J(z^k)-J(z^{k+})
    \leq\left(\frac{2L}{\sigma\sqrt{w}}+\frac{L}{2}\right)\|r^k\|^2.
\]
Together with \eqref{eq:gap_at_base_point}, this gives
\begin{equation}
    V_k
    \leq\mu\|r^k\|^2,
    \qquad
    \mu=\frac{2L}{\sigma\sqrt{w}}+\frac{L}{2}+\frac{2}{\sigma w}.
    \label{eq:step2_gap}
\end{equation}

\textbf{Step 3: Linear convergence.}
Combining \eqref{eq:step1_decrease} and \eqref{eq:step2_gap},
\[
    V_{k+1}
    \leq\left(1-\frac{\lambda}{\mu}\right)V_k.
\]
Let $\rho:=\lambda/\mu\in(0,1)$. Then
\begin{equation}
    V_k\leq V_0(1-\rho)^k,
    \label{eq:Q_linear_objective}
\end{equation}
so the objective values converge Q-linearly. By quadratic growth,
\begin{equation}
    \operatorname{dist}(z^k,\mathcal{Z}^*)
    \leq\sqrt{\frac{2V_k}{\sigma}}
    \leq\sqrt{\frac{2V_0}{\sigma}}\,(\sqrt{1-\rho})^k.
    \label{eq:R_linear_distance}
\end{equation}
Thus the distance to the solution set converges R-linearly.

Finally, \eqref{eq:step1_decrease} and \eqref{eq:Q_linear_objective} imply
\[
    \|r^k\|^2
    \leq\frac{V_k}{\lambda}
    \leq\frac{V_0}{\lambda}(1-\rho)^k.
\]
Using \eqref{eq:H_bound},
\[
    \|z^{k+1}-z^k\|
    \leq\sqrt{2(1+c_2)}\,\beta\sqrt{\frac{V_0}{\lambda}}
    (\sqrt{1-\rho})^k.
\]
Therefore $\sum_{k=0}^\infty\|z^{k+1}-z^k\|<\infty$, so $\{z^k\}$ is Cauchy and converges to some $z^\infty$. Since \eqref{eq:R_linear_distance} tends to zero and $\mathcal{Z}^*$ is closed, $z^\infty\in\mathcal{Z}^*$. Moreover,
\[
    \|z^k-z^\infty\|
    \leq\sum_{j=k}^\infty\|z^{j+1}-z^j\|
    \leq
    \frac{\sqrt{2(1+c_2)}\,\beta}{1-\sqrt{1-\rho}}
    \sqrt{\frac{V_0}{\lambda}}
    (\sqrt{1-\rho})^k,
\]
which proves the R-linear convergence of the iterates themselves.
\end{proof}
\section{Numerical Experiments}
\label{sec:experiments}

In this section, we validate the effectiveness of our M-IOT framework and the convergence properties of our algorithm on both synthetic and real-world datasets. We demonstrate the model's ability to recover ground-truth costs from noisy observations, as well as its applicability to urban mobility modeling and e-commerce recommendation. Detailed descriptions of data generation and preprocessing pipelines for real-world experiments are provided in the supplementary material.

\subsection{Experimental Setup and Implementation Details}
\label{subsec:setup}

First of all, we detail the hardware environment, implementation specifics, and hyperparameter settings used throughout the numerical experiments. All experiments were conducted on a laptop with an 11th Gen Intel(R) Core(TM) i7-11800H @ 2.30GHz CPU. The code were implemented in Python 3.13.2.  We solve the M-IOT convex formulation using the proposed BCDwAE (Algorithm \ref{alg:PRAABCD}).  For the iteration of BCD, the dual potentials ($u, v$) are updated via closed-form solutions, while the cost matrix $C$ is updated using Projected Gradient Descent with a backtracking line search (Armijo line search) to determine the step size. To avoid instability during the initial transient phase, we run the standard BCD (without acceleration) for a warm-start period. Acceleration is activated only after 10 iterations and when the relative KKT residual drops below $1.0$. We set the memory depth $m=8$, the mixing parameter $\beta=1.0$, and the regularization parameter $\delta=10^{-8}$ to stabilize the least-squares subproblem. A restart is triggered every 50 iterations to clear stale history.  We monitor convergence using relative KKT residual, $\mathcal{E}_{\text{KKT}}$, which is defined as follows. Let $X^{l}$ be the transport plan induced by the current iterates. We define the normalized marginal residuals as $\mathcal{E}_{\mu} = (\sum_l \|X^{l}\mathbf{1} - \mu^l\|_2) / (\|\mu\| + \epsilon)$ and $\mathcal{E}_{\nu} = (\sum_l \|X^{l\top}\mathbf{1} - \nu^l\|_2) / (\|\nu\| + \epsilon)$. The residual of the cost matrix is measured by the norm of the fixed-point residual of the projected gradient step: $\mathcal{E}_{C} = \|C - \text{Proj}_{\mathcal{C}}(C - \nabla_C J)\|_F / (\|C\|_F + 1)$.
Here $\epsilon$ is a small constant for numerical stability. Now we can define the root-mean-square error $\mathcal{E}_{\text{KKT}} = \sqrt{(\mathcal{E}_{\mu}^2 + \mathcal{E}_{\nu}^2 + \mathcal{E}_{C}^2)/3} $.  The algorithm terminates when the relative KKT residual $\mathcal{E}_{\text{KKT}}$  falls below a tolerance of $\tau = 10^{-6}$, or when the maximum number of iterations (set to 500) is reached. Throughout this section, we choose the entropy regularization parameter $\gamma = 1$ for all the experiments. Unless otherwise specified, these settings are consistent across both synthetic and real-world experiments.

\subsection{Experiments on Synthetic Data}

We generate $K$ observation pairs $\{(\mu^l, \nu^l)\}_{l=1}^K$ based on a fixed ground truth cost $C_{\text{true}} \in \mathbb{R}^{n \times n}$ with $K=20$ and $n=20$. 
We first sample $n$ points $\{p_1, \dots, p_n\}$ uniformly from the unit square $[0, 1]^2$ and then compute the ground-truth cost matrix $C_{\text{true}} \in \mathbb{R}^{n \times n}$ via squared Euclidean distances: $C_{ij} = \|p_i - p_j\|_2^2$. We generate the marginals $\{(\mu^l, \nu^l)\}_{l=1}^K$ under two distinct scenarios. The first scenario simulates continuous evolution using interpolated marginals. Specifically, we generate random starting marginals $(\mu_{\text{start}}, \nu_{\text{start}})$ and ending marginals $(\mu_{\text{end}}, \nu_{\text{end}})$. The $l$-th observation is generated via linear interpolation:
    \begin{equation}
        \mu^l = (1 - \lambda_l)\mu_{\text{start}} + \lambda_l \mu_{\text{end}}, \quad \nu^l = (1 - \lambda_l)\nu_{\text{start}} + \lambda_l \nu_{\text{end}},
    \end{equation}
    where $\lambda_l = \frac{l-1}{K-1}$ for $l=1, \dots, K$.
For the second scenario, we use random marginals as comparison to the first scenario. Each pair $(\mu^l, \nu^l)$ was generated independently by sampling from a uniform distribution over $[0, 1]^n$ and normalizing such that $\sum_i \mu^l_i = \sum_j \nu^l_j = 1$.

The observed transport plans $\hat{X}^l$ were generated by perturbing the theoretical optimal plans. Let $X^l_{\text{perfect}} = X^*(C_{\text{true}}, \mu^l, \nu^l)$ be the solution to the forward entropy-regularized OT problem with $\gamma=1.0$, obtained using the Sinkhorn algorithm\cite{cuturi2013sinkhorn}. We applied multiplicative noise:
\begin{equation}
    \hat{X}^l = X^l_{\text{perfect}} \odot \exp(\delta \cdot \mathcal{N}),
\end{equation}
where $\odot$ denotes element-wise multiplication, $\mathcal{N}$ is a matrix of standard Gaussian noise $\mathcal{N}_{ij} \sim N(0, 1)$, and $\delta=0.01$ is the noise level. Note that after noise injection, the marginals of $\hat{X}^l$ do not strictly match $\mu^l$ and $\nu^l$, testing the model's robustness to constraint violation. We compare M-IOT against two baselines: (1) Single Observation ($K=1$), and (2) Averaged Baseline (averaging $K$ independently solved IOT costs).

\begin{figure}[h!]
    \centering
    \includegraphics[width=\textwidth]{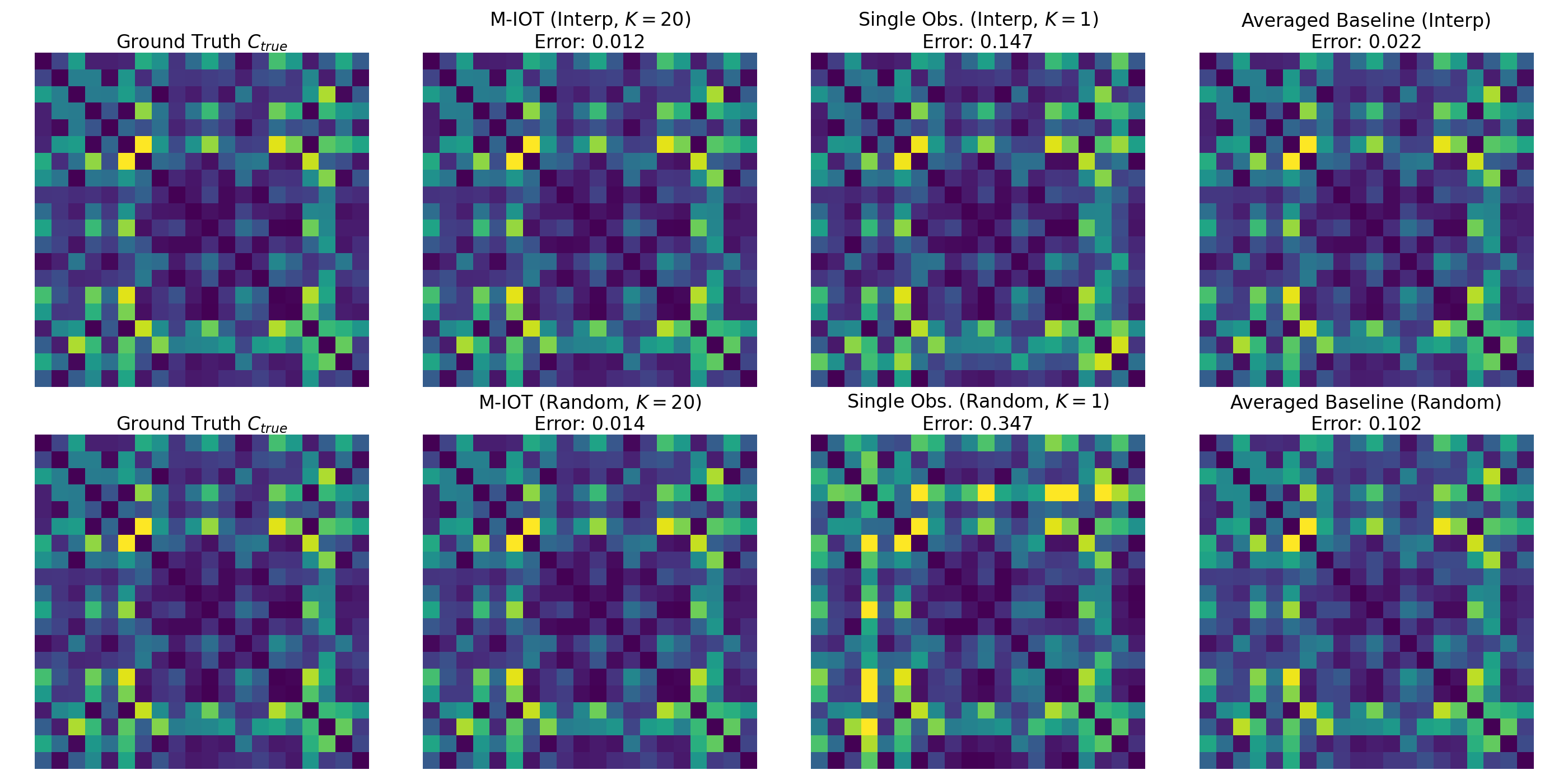}
    \caption{Qualitative and quantitative comparison of cost matrix recovery. }
    \label{fig:synth_recovery}
\end{figure}

The results are displayed in Figure~\ref{fig:synth_recovery}, which also visualizes the recovered cost matrices. The proposed M-IOT framework successfully recovers the ground-truth structure with high fidelity, outperforming the baselines in both the Interpolated and Random scenarios.  M-IOT achieves a relative error of \textbf{0.012} in the interpolated scenario, which is an order of magnitude lower than the single-observation baseline (0.147) and approximately half the error of the averaged baseline (0.022). 

\subsubsection{Impact of Observation Count ($K$) and Convergence}

We further investigate the consistency and convergence properties of the framework. In a separate ablation study, we examine the recovery error as the number of observations $K$ increases from 1 to 100. Figure \ref{fig:synth_ablation} shows an overall downward trend in the relative error for both marginal-generation scenarios, with minor fluctuations. This empirical trend suggests that integrating more snapshots of the transport process generally improves recovery of the ground-truth cost. 

\begin{figure}[h!]
    \centering
    \begin{minipage}{0.48\textwidth}
        \centering
        \includegraphics[width=\linewidth]{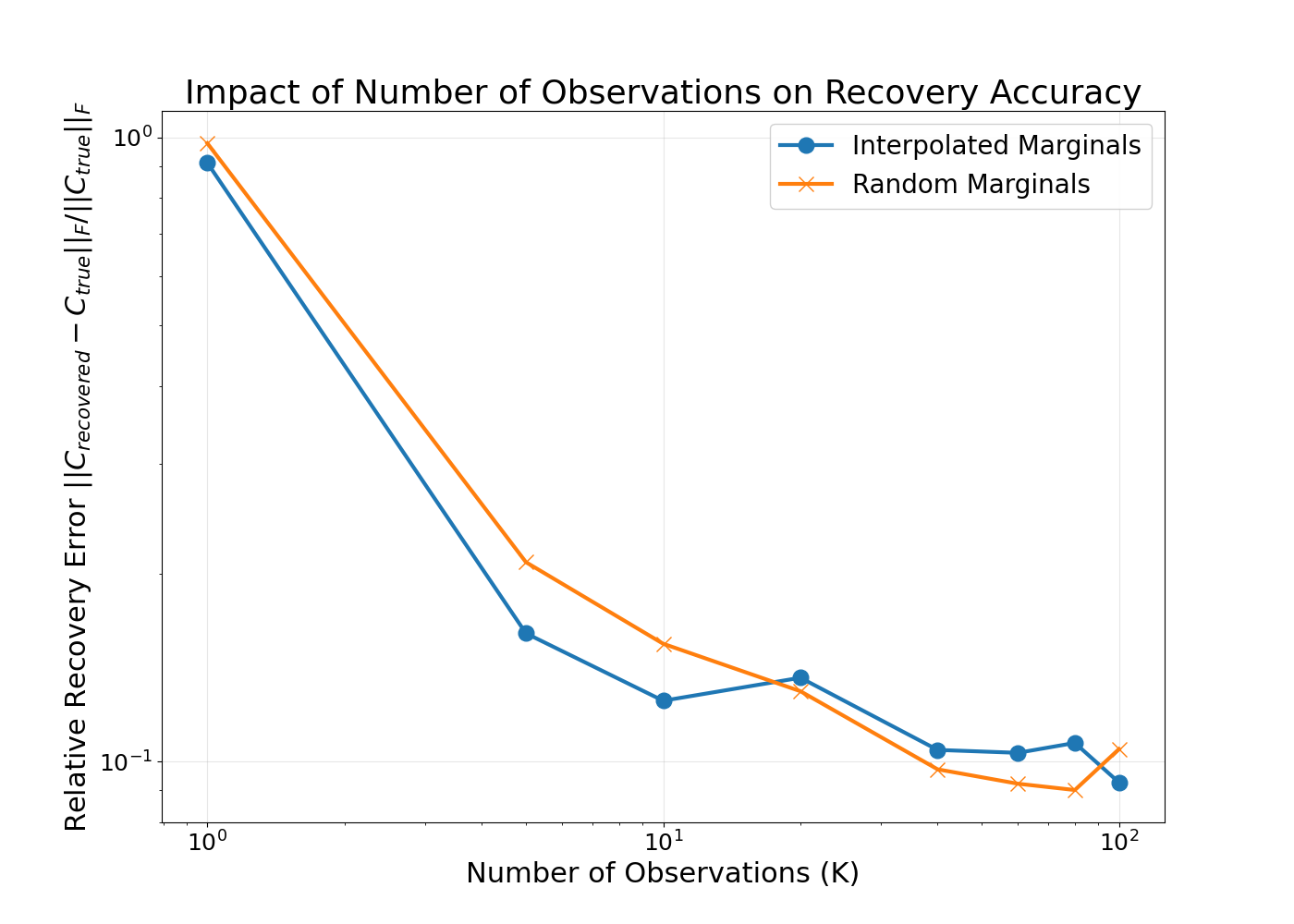}
        \caption{Recovery error vs. number of observations $K$.}
        \label{fig:synth_ablation}
    \end{minipage}
    \hfill
    \begin{minipage}{0.48\textwidth}
        \centering
        \includegraphics[width=\linewidth]{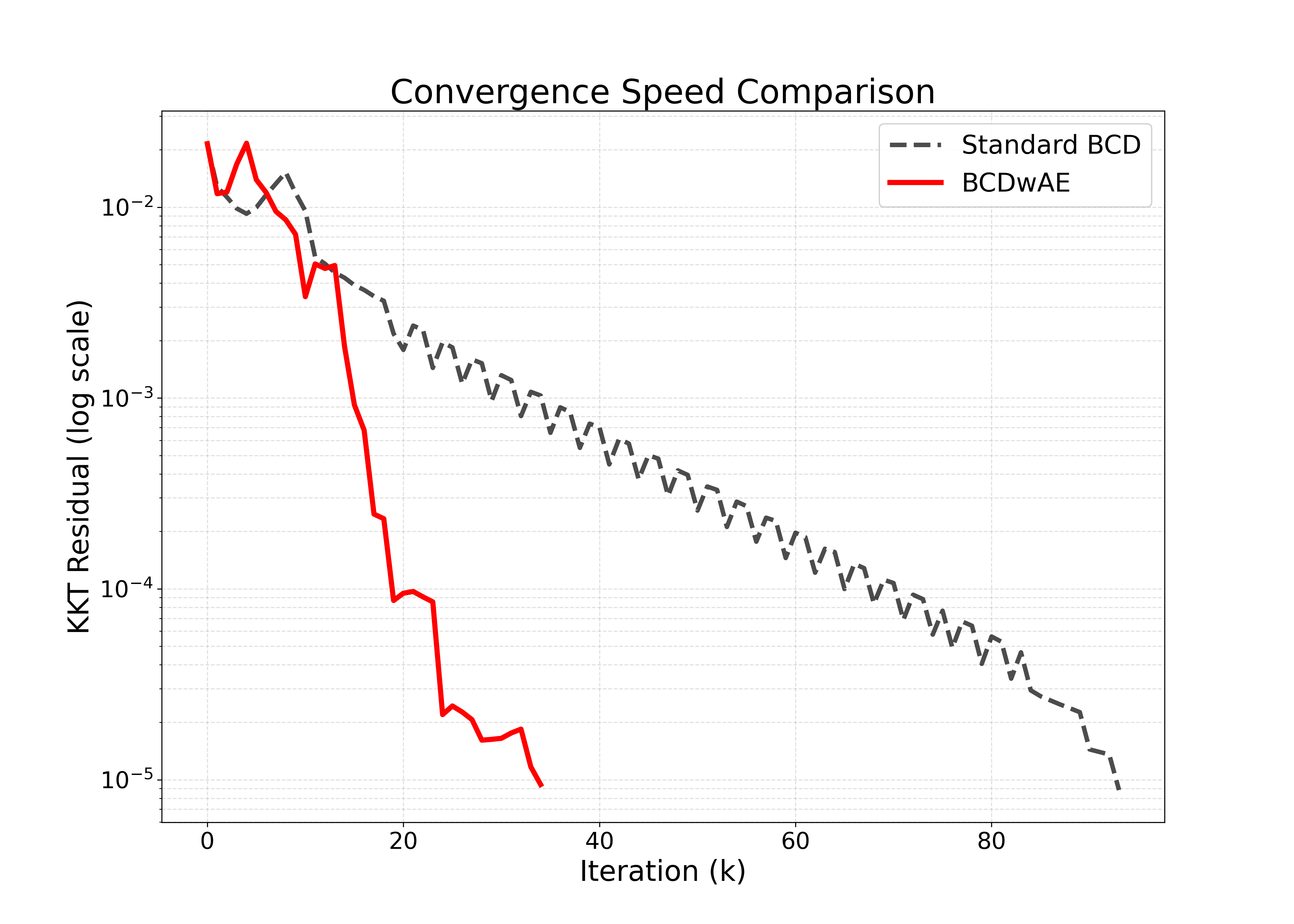}
        \caption{Acceleration of BCDwAE compared to the original method.}
        \label{fig:synth_convergence}
    \end{minipage}
\end{figure}

Then, we examine the convergence properties of the proposed algorithm. Figure \ref{fig:synth_convergence} plots the KKT residual against the iteration number on a logarithmic scale. Both the standard Block Coordinate Descent (BCD) and our Block Coordinate Descent with Anderson-type Extrapolation (BCDwAE) algorithm exhibit a linear decrease in residual, validating the Q-linear convergence rate established in Theorem \ref{thm:linear_convergence}. Furthermore, the accelerated algorithm demonstrates a steeper descent trajectory compared to the standard BCD. 

\subsection{Application: Urban Mobility (NYC Taxi Data)}
We apply M-IOT to model the urban mobility structure using the NYC Taxi dataset\footnote{https://www.kaggle.com/c/nyc-taxi-trip-duration/data}. The city is discretized into $n=100$ zones, and trip data is partitioned into $K=6$ distinct time-window scenarios (e.g., Weekday AM, Weekend Night) to capture temporal heterogeneity. For each scenario $l$, the marginals $(\mu^l, \nu^l)$ represent pickup/drop-off distributions, and the plan $\hat{X}^l$ is the aggregated flow matrix. We aim to learn a single time-invariant cost matrix $C$ that rationalizes these diverse flow patterns. To demonstrate the advantage of our data-driven approach, we benchmark against an average of independent IOT solutions and the classical Gravity Model \cite{wilson1967statistical}, the standard parametric formulation in transportation geography. The Gravity Model assumes that the flow $T_{ij}$ between zones decays with physical distance $d_{ij}$ according to a power law:$T_{ij} \propto O_i^\alpha D_j^\beta d_{ij}^{-\gamma}$, where $O_i, D_j$ are marginal outflows and inflows, and $\alpha, \beta, \gamma$ are parameters fitted via regression. 

We evaluate the prediction accuracy on test data using three complementary metrics: RMSE, Pearson Correlation, and
Common Part of Commuters (CPC) \cite{lenormand2012universal}, which is a specialized metric for transport flows defined as:
    $$
    \text{CPC}(X, \hat{X}) = \frac{2 \sum_{i,j} \min(X_{ij}, \hat{X}_{ij})}{\sum_{i,j} X_{ij} + \sum_{i,j} \hat{X}_{ij}}.
    $$
    The CPC takes values in $[0, 1]$, providing a strict measure of the intersection between the predicted and observed flow distributions (where 1 indicates perfect overlap).

\begin{table}[h!]
\centering
\caption{Performance Comparison on NYC Taxi Test Set (Averaged over 6 Scenarios)}
\label{tab:nyc_results}
\begin{tabular}{lccc}
\toprule
\textbf{Model} & \textbf{RMSE} $\downarrow$ & \textbf{Correlation} $\uparrow$ & \textbf{CPC} $\uparrow$ \\
\midrule
\textbf{M-IOT (Symmetric)} & \textbf{4.45}& \textbf{0.9989} & \textbf{0.9566}\\
Naive Average IOT & 4.76& 0.9988 & 0.9489\\
Gravity Model (Baseline) & 37.12 & 0.9076 & 0.6568 \\
\bottomrule
\end{tabular}
\end{table}
Table \ref{tab:nyc_results} summarizes the predictive performance averaged over the six test scenarios. The M-IOT framework achieves the best results across all metrics, with an RMSE of \textbf{4.45} and a CPC of \textbf{0.9566}. 
The proposed method  outperforms the parametric Gravity Model, which yields a much higher RMSE (37.12) and a lower CPC (0.6568), and outperforms the  Average IOT baseline (RMSE 4.76).   

\subsection{Application: E-commerce Recommendation}

We now apply M-IOT to learn latent consumer preferences from the "E-Commerce" dataset\footnote{https://www.kaggle.com/datasets/carrie1/ecommerce-data }. Here, purchasing behavior is modeled as a transport problem where mass moves from users to products. 
We propose a bilinear parameterization:
$$
C = G^\top B V,
$$
where $G \in \mathbb{R}^{d \times m}$ and $V \in \mathbb{R}^{d \times n}$ are fixed feature matrices for users and items, respectively, and $B \in \mathbb{R}^{d \times d}$ is the learnable interaction matrix. The optimization is performed over $B$, reducing the parameter space from $m \times n$ to $d^2$. We use the "E-Commerce" dataset, treating monthly transaction logs (Dec 2010 – Nov 2011) as $K=12$ distinct observations.  Users are aggregated into $m=50$ segments based on Recency, Frequency, and Monetary (RFM) features to form the source distribution $\mu$. The target $\nu$ consists of the top $n=100$ products. The feature matrices $G$ and $V$ are constructed from the centroids of user clusters and item attributes ($d=3$). We evaluate the learned structure on test data using  Precision@K and Recall@K, comparing against Matrix Factorization (ALS)\cite{hu2008collaborative}\footnote{Implemented via the Python \texttt{implicit} library}, Averaged IOT and Popularity Model\cite{cremonesi2010performance} which recommends the most popular commodity. Figures  \ref{fig:precision_k_curve} and \ref{fig:recall_k_curve}  illustrate the recommendation performance. We can see that our M-IOT model outperforms both the ALS and Popularity baselines across all list sizes ($K$). 

\begin{figure}[h!]
    \centering
    \begin{minipage}{0.48\textwidth}
        \centering
        \includegraphics[width=\linewidth]{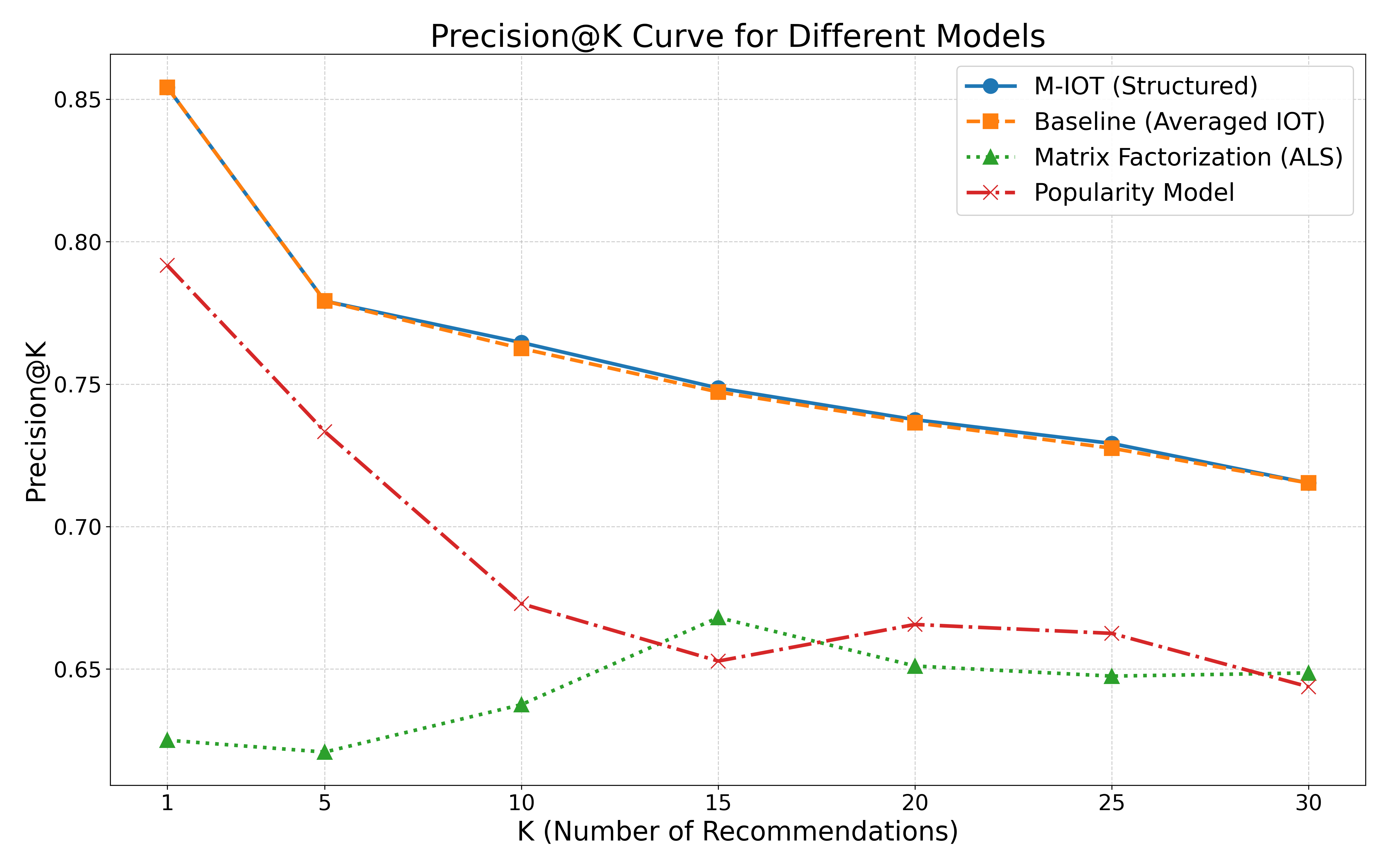}
        \caption{Precision@K curves. }
        \label{fig:precision_k_curve}
    \end{minipage}
    \hfill
    \begin{minipage}{0.48\textwidth}
        \centering
        \includegraphics[width=\linewidth]{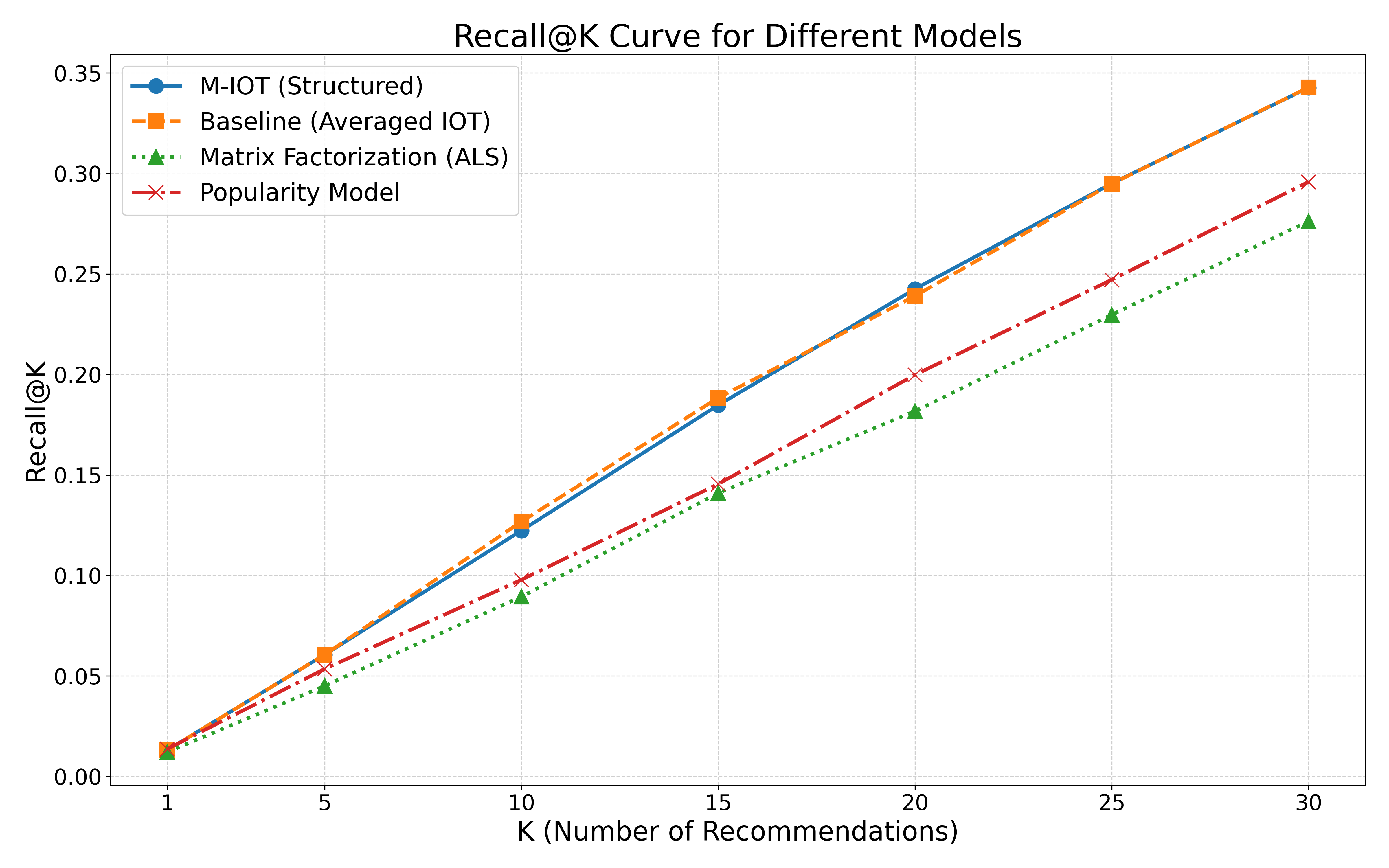}
        \caption{Recall@K curves. }
        \label{fig:recall_k_curve}
    \end{minipage}
\end{figure}

\section{Conclusion and Future Work}
\label{sec:conclusion}

In this paper, we introduce the Multi-Observation Inverse Optimal Transport (M-IOT) framework, a robust approach for inferring a universal cost function from multiple transport plans. We reformulate the bilevel inverse problem as an equivalent single-level convex optimization problem, and analyze its theoretical properties, establishing the existence of solutions and characterizing the structure of the optimal solution set. To address the computational challenges, we propose a Block Coordinate Descent with Anderson-type Extrapolation algorithm. Extensive numerical experiments on synthetic data, urban mobility flows, and e-commerce recommendations demonstrate that M-IOT achieves lower recovery errors and successfully extracts stable, invariant structural patterns from noisy, multi-view data.

Future work includes extending M-IOT beyond entropy regularization to more general regularization, modeling time-varying cost structures in dynamic settings, and improving scalability in high-dimensional problems through stochastic methods or low-rank approximations.


\bibliographystyle{siamplain}
\bibliography{references}

@article{wilson1967statistical,
  title={A statistical theory of spatial distribution models},
  author={Wilson, AG},
  journal={Transportation Research},
  volume={1},
  number={3},
  pages={253--269},
  year={1967},
  publisher={Elsevier}
}

@inproceedings{chiu2022discrete,
  title={Discrete probabilistic inverse optimal transport},
  author={Chiu, Wei-Ting and Wang, Pei and Shafto, Patrick},
  booktitle={International Conference on Machine Learning},
  pages={3925--3946},
  year={2022},
  organization={PMLR}
}

@inproceedings{wang2023self,
  title={Self-supervised video summarization guided by semantic inverse optimal transport},
  author={Wang, Yutong and Xu, Hongteng and Luo, Dixin},
  booktitle={Proceedings of the 31st ACM International Conference on Multimedia},
  pages={6611--6622},
  year={2023}
}

@inproceedings{yu2022explainable,
  title={Explainable legal case matching via inverse optimal transport-based rationale extraction},
  author={Yu, Weijie and Sun, Zhongxiang and Xu, Jun and Dong, Zhenhua and Chen, Xu and Xu, Hongteng and Wen, Ji-Rong},
  booktitle={Proceedings of the 45th international ACM SIGIR conference on research and development in information retrieval},
  pages={657--668},
  year={2022}
}

@article{carlier2023sista,
  title={Sista: learning optimal transport costs under sparsity constraints},
  author={Carlier, Guillaume and Dupuy, Arnaud and Galichon, Alfred and Sun, Yifei},
  journal={Communications on Pure and Applied Mathematics},
  volume={76},
  number={9},
  pages={1659--1677},
  year={2023},
  publisher={Wiley Online Library}
}

@article{persiianov2024inverse,
  title={Inverse Entropic Optimal Transport Solves Semi-supervised Learning via Data Likelihood Maximization},
  author={Persiianov, Mikhail and Asadulaev, Arip and Andreev, Nikita and Starodubcev, Nikita and Baranchuk, Dmitry and Kratsios, Anastasis and Burnaev, Evgeny and Korotin, Alexander},
  journal={arXiv preprint arXiv:2410.02628},
  year={2024}
}

@article{dupuy2019estimating,
  title={Estimating matching affinity matrices under low-rank constraints},
  author={Dupuy, Arnaud and Galichon, Alfred and Sun, Yifei},
  journal={Information and Inference: A Journal of the IMA},
  volume={8},
  number={4},
  pages={677--689},
  year={2019},
  publisher={Oxford University Press}
}

@article{stuart2020inverse,
  title={Inverse optimal transport},
  author={Stuart, Andrew M and Wolfram, Marie-Therese},
  journal={SIAM Journal on Applied Mathematics},
  volume={80},
  number={1},
  pages={599--619},
  year={2020},
  publisher={SIAM}
}

@article{andrade2023sparsistency,
  title={Sparsistency for inverse optimal transport},
  author={Andrade, Francisco and Peyr{\'e}, Gabriel and Poon, Clarice},
  journal={arXiv preprint arXiv:2310.05461},
  year={2023}
}

@article{li2019learning,
  title={Learning to match via inverse optimal transport},
  author={Li, Ruilin and Ye, Xiaojing and Zhou, Haomin and Zha, Hongyuan},
  journal={Journal of machine learning research},
  volume={20},
  number={80},
  pages={1--37},
  year={2019}
}

@article{ma2020learning,
  title={Learning cost functions for optimal transport},
  author={Ma, Shaojun and Sun, Haodong and Ye, Xiaojing and Zha, Hongyuan and Zhou, Haomin},
  journal={arXiv preprint arXiv:2002.09650},
  year={2020}
}

@article{gonzalez2024identifiability,
  title={Identifiability of the Optimal Transport Cost on Finite Spaces},
  author={Gonz{\'a}lez-Sanz, Alberto and Groppe, Michel and Munk, Axel},
  journal={arXiv preprint arXiv:2410.23146},
  year={2024}
}

@article{gonzalez2024nonlinear,
  title={Nonlinear inverse optimal transport: Identifiability of the transport cost from its marginals and optimal values},
  author={Gonz{\'a}lez-Sanz, Alberto and Groppe, Michel and Munk, Axel},
  journal={SIAM Journal on Mathematical Analysis},
  volume={56},
  number={6},
  pages={7808--7829},
  year={2024},
  publisher={SIAM}
}

@article{cuturi2013sinkhorn,
  title={Sinkhorn distances: Lightspeed computation of optimal transport},
  author={Cuturi, Marco},
  journal={Advances in neural information processing systems},
  volume={26},
  year={2013}
}

@book{villani2008optimal,
  title={Optimal transport: old and new},
  author={Villani, C{\'e}dric and others},
  volume={338},
  year={2008},
  publisher={Springer}
}

@article{courty2016optimal,
  title={Optimal transport for domain adaptation},
  author={Courty, Nicolas and Flamary, R{\'e}mi and Tuia, Devis and Rakotomamonjy, Alain},
  journal={IEEE transactions on pattern analysis and machine intelligence},
  volume={39},
  number={9},
  pages={1853--1865},
  year={2016},
  publisher={IEEE}
}

@article{kolouri2017optimal,
  title={Optimal mass transport: Signal processing and machine-learning applications},
  author={Kolouri, Soheil and Park, Se Rim and Thorpe, Matthew and Slepcev, Dejan and Rohde, Gustavo K},
  journal={IEEE signal processing magazine},
  volume={34},
  number={4},
  pages={43--59},
  year={2017},
  publisher={IEEE}
}

@phdthesis{genevay2019entropy,
  title={Entropy-regularized optimal transport for machine learning},
  author={Genevay, Aude},
  year={2019},
  school={Universit{\'e} Paris sciences et lettres}
}

@book{rockafellar1997convex,
  title={Convex analysis},
  author={Rockafellar, R Tyrrell},
  volume={28},
  year={1997},
  publisher={Princeton university press}
}

@inproceedings{cremonesi2010performance,
  title={Performance of recommender algorithms on top-n recommendation tasks},
  author={Cremonesi, Paolo and Koren, Yehuda and Turrin, Roberto},
  booktitle={Proceedings of the fourth ACM conference on Recommender systems},
  pages={39--46},
  year={2010}
}

@article{liu2019learning,
  title={Learning transport cost from subset correspondence},
  author={Liu, Ruishan and Balsubramani, Akshay and Zou, James},
  journal={arXiv preprint arXiv:1909.13203},
  year={2019}
}

@article{andrade2025learning,
  title={Learning from Samples: Inverse Problems over measures via Sharpened Fenchel-Young Losses},
  author={Andrade, Francisco and Peyr{\'e}, Gabriel and Poon, Clarice},
  journal={arXiv preprint arXiv:2505.07124},
  year={2025}
}

@inproceedings{shi2023understanding,
  title={Understanding and generalizing contrastive learning from the inverse optimal transport perspective},
  author={Shi, Liangliang and Zhang, Gu and Zhen, Haoyu and Fan, Jintao and Yan, Junchi},
  booktitle={International conference on machine learning},
  pages={31408--31421},
  year={2023},
  organization={PMLR}
}

@article{samaran2024scconfluence,
  title={scConfluence: single-cell diagonal integration with regularized Inverse Optimal Transport on weakly connected features},
  author={Samaran, Jules and Peyr{\'e}, Gabriel and Cantini, Laura},
  journal={Nature Communications},
  volume={15},
  number={1},
  pages={7762},
  year={2024},
  publisher={Nature Publishing Group UK London}
}

@article{elvander2025mixtures,
  title={Mixtures of ensembles: System separation and identification via optimal transport},
  author={Elvander, Filip and Haasler, Isabel},
  journal={arXiv preprint arXiv:2503.13362},
  year={2025}
}

@article{mascherpa2025convex,
  title={A convex approach for Markov chain estimation from aggregate data via inverse optimal transport},
  author={Mascherpa, Michele and Ringh, Axel and Taghvaei, Amirhossein and Karlsson, Johan},
  journal={arXiv preprint arXiv:2511.16458},
  year={2025}
}

@article{bauschke1996projection,
  title={On projection algorithms for solving convex feasibility problems},
  author={Bauschke, Heinz H and Borwein, Jonathan M},
  journal={SIAM review},
  volume={38},
  number={3},
  pages={367--426},
  year={1996},
  publisher={SIAM}
}

@article{bauschke1999strong,
  title={Strong conical hull intersection property, bounded linear regularity, Jameson’s property (G), and error bounds in convex optimization},
  author={Bauschke, Heinz H and Borwein, Jonathan M and Li, Wu},
  journal={Mathematical Programming},
  volume={86},
  number={1},
  pages={135--160},
  year={1999},
  publisher={Springer}
}

@article{wei2021stochastic,
  title={Stochastic Anderson mixing for nonconvex stochastic optimization},
  author={Wei, Fuchao and Bao, Chenglong and Liu, Yang},
  journal={Advances in Neural Information Processing Systems},
  volume={34},
  pages={22995--23008},
  year={2021}
}

@article{polyak1963gradient,
  title={Gradient methods for the minimisation of functionals},
  author={Polyak, Boris T},
  journal={USSR Computational Mathematics and Mathematical Physics},
  volume={3},
  number={4},
  pages={864--878},
  year={1963},
  publisher={Elsevier}
}

@article{anderson1965iterative,
  title={Iterative procedures for nonlinear integral equations},
  author={Anderson, Donald G},
  journal={Journal of the ACM (JACM)},
  volume={12},
  number={4},
  pages={547--560},
  year={1965},
  publisher={ACM New York, NY, USA}
}

@article{lenormand2012universal,
  title={A universal model of commuting networks},
  author={Lenormand, Maxime and Huet, Sylvie and Gargiulo, Floriana and Deffuant, Guillaume},
  year={2012},
  publisher={Public Library of Science San Francisco, USA}
}

@inproceedings{hu2008collaborative,
  title={Collaborative filtering for implicit feedback datasets},
  author={Hu, Yifan and Koren, Yehuda and Volinsky, Chris},
  booktitle={2008 Eighth IEEE international conference on data mining},
  pages={263--272},
  year={2008},
  organization={Ieee}
}

@book{bertsekas2009convex,
  title={Convex optimization theory},
  author={Bertsekas, Dimitri},
  volume={1},
  year={2009},
  publisher={Athena Scientific}
}

@inproceedings{liao2024error,
  title={Error bounds, PL condition, and quadratic growth for weakly convex functions, and linear convergences of proximal point methods},
  author={Liao, Feng-Yi and Ding, Lijun and Zheng, Yang},
  booktitle={6th Annual Learning for Dynamics \& Control Conference},
  pages={993--1005},
  year={2024},
  organization={PMLR}
}

@article{10.1088/1361-6420/ae5ace,
	author={Bao, Chenglong and Li, Zanyu and Yang, Yunan},
	title={Well-posedness and efficient algorithms for inverse optimal transport with Bregman regularization},
	journal={Inverse Problems},
	url={http://iopscience.iop.org/article/10.1088/1361-6420/ae5ace},
	year={2026}
}
\end{document}


\maketitle

\section{Algebraic Structure and Spectral Analysis of Operator A}\label{app:operatorA}

In this section, we provide an explicit construction of $A$ and conduct a complete spectral analysis of the operator $A^TA$.

\subsection{Explicit Operator Representation}
Let us first provide the explicit matrix representation of $A$ and the closed-form actions of its related operators. The operator $A: \R^{2nK+n^2} \to \R^{Kn^2}$ maps the concatenated vector of variables $z=(u^1, \dots, u^K, v^1, \dots, v^K, \vecop(C))^T$ to the vector of exponential arguments $y$, where $y_{lij} = u_i^l + v_j^l - C_{ij}$. We use column-wise vectorization within each observation block.

\textbf{Matrix Form of A.} Using the Kronecker product $\otimes$, the block matrix $A = [A_u | A_v | A_C]$ has the following components:
\begin{align}
    A_u &= I_K \otimes (\ones_n \otimes I_n), \\
    A_v &= I_K \otimes (I_n \otimes \ones_n), \\
    A_C &= -(\ones_K \otimes I_{n^2}).
\end{align}

\textbf{Action of $A^T$.} Given $y \in \R^{Kn^2}$, its adjoint action $w = A^T y$ produces $w=(w_u, w_v, w_C)$, where:
\begin{align}
    (w_u^l)_i = \sum_{j=1}^n y_{lij}, \quad (w_v^l)_j = \sum_{i=1}^n y_{lij}, \quad (w_C)_{ij} = - \sum_{l=1}^K y_{lij}.
\end{align}

\textbf{Action of $AA^T$.} Given $y \in \R^{Kn^2}$, the action $w = AA^T y$ is given by:
\begin{equation}
    w_{lij} = \left(\sum_{j'=1}^n y_{lij'}\right) + \left(\sum_{i'=1}^n y_{li'j}\right) + \left(\sum_{l'=1}^K y_{l'ij}\right).
\end{equation}

\textbf{Action of $A^TA$.} Given $z=(u,v,C)$, the action $w = A^TAz$ is given by:
\begin{align}
    (w_u^l)_i &= n u_i^l + \sum_{j=1}^n v_j^l - \sum_{j=1}^n C_{ij}, \\
    (w_v^l)_j &= n v_j^l + \sum_{i=1}^n u_i^l - \sum_{i=1}^n C_{ij}, \\
    (w_C)_{ij} &= K C_{ij} - \sum_{l=1}^K u_i^l - \sum_{l=1}^K v_j^l.
\end{align}

\subsection{The Spectrum of $A^TA$}
We show that the eigenvalues of $A^TA$ are integers determined strictly by $n$ and $K$.
\begin{theorem}[Spectrum of $A^TA$]
\label{thm:spectrum}
For the M-IOT problem with a square cost matrix and $n,K\geq2$, the set of all eigenvalues of the operator $A^TA$ is $\{0, n, K, 2n, n+K, 2n+K\}$.
\end{theorem}

\begin{proof}
Let $z = (u^1, \dots, u^K, v^1, \dots, v^K, C)$ be an eigenvector corresponding to an eigenvalue $\lambda$. The characteristic equations $A^T A z = \lambda z$ are given by:
\begin{align}
(n - \lambda)u_i^l &= \sum_{j=1}^n C_{ij} - \sum_{j=1}^n v_j^l, \label{eq:eig_u}\\
(n - \lambda)v_j^l &= \sum_{i=1}^n C_{ij} - \sum_{i=1}^n u_i^l, \label{eq:eig_v}\\
(K - \lambda)C_{ij} &= \sum_{l=1}^K u_i^l + \sum_{l=1}^K v_j^l. \label{eq:eig_C}
\end{align}
We define the following aggregate quantities: $S_u^l = \sum_i u_i^l$, $S_v^l = \sum_j v_j^l$, $R_i(C) = \sum_j C_{ij}$, $C_j(C) = \sum_i C_{ij}$, and $S_C = \sum_{i,j} C_{ij} = \sum_i R_i(C) = \sum_j C_j(C)$. First of all, summing \eqref{eq:eig_u} over $i$ and \eqref{eq:eig_v} over $j$ yields:
\begin{align}
(n-\lambda)S_u^l &= S_C - n S_v^l, \label{eq:L1_1}\\
(n-\lambda)S_v^l &= S_C - n S_u^l. \label{eq:L1_2}
\end{align}
Subtracting \eqref{eq:L1_2} from \eqref{eq:L1_1} gives $-\lambda(S_u^l - S_v^l) = 0$. For any non-zero eigenvalue $\lambda \neq 0$, we must have $S_u^l = S_v^l$. Let $S^l := S_u^l = S_v^l$.
In this case \eqref{eq:L1_1} and \eqref{eq:L1_2}  simplify to:
\begin{equation}
(2n - \lambda)S^l = S_C. \label{eq:L1_main}
\end{equation}
Summing \eqref{eq:eig_C} over all $i,j$ gives $(K-\lambda)S_C = 2n \sum_{l=1}^K S^l$. We consider two cases derived from \eqref{eq:L1_main}:
\begin{enumerate}
    \item \textbf{Case $\lambda = 2n$:} Equation \eqref{eq:L1_main} implies $S_C = 0$. The cost equation becomes $(K-2n) \cdot 0 = 2n \sum_l S^l$, implying $\sum_l S^l = 0$. Thus, $\lambda = 2n$ is an eigenvalue associated with eigenvectors where sums satisfy $S_u^l = S_v^l$, $\sum_l S^l = 0$, and $S_C = 0$.
    \item \textbf{Case $\lambda \neq 2n$:} We can write $S^l = S_C / (2n-\lambda)$. Substituting this into the summed cost equation yields:
    $$
    (K-\lambda)S_C = 2n \sum_{l=1}^K \frac{S_C}{2n-\lambda} = \frac{2nK}{2n-\lambda}S_C.
    $$
    Assuming $S_C \neq 0$, we divide by $S_C$ to get $(K-\lambda)(2n-\lambda) = 2nK$. The roots are $\lambda = 0$ (trivial) and $\lambda = 2n + K$. Thus, $\lambda = 2n+K$ is an eigenvalue.
\end{enumerate}

We now assume $S_u^l = S_v^l = 0$ for all $l$, and $S_C = 0$.
Equations \eqref{eq:eig_u} and \eqref{eq:eig_v} simplify to $(n-\lambda)u_i^l = R_i(C)$ and $(n-\lambda)v_j^l = C_j(C)$.
Let $\bar{u}_i = \sum_l u_i^l$ and $\bar{v}_j = \sum_l v_j^l$. Summing the simplified equations over $l$:
\begin{equation}
(n-\lambda)\bar{u}_i = K R_i(C), \quad (n-\lambda)\bar{v}_j = K C_j(C). \label{eq:L2_uv}
\end{equation}
Note that \eqref{eq:eig_C} simplifies into $(K-\lambda)C_{ij} = \bar{u}_i + \bar{v}_j$. Summing this over $j$ gives:
$$
(K-\lambda)R_i(C) = n \bar{u}_i + \sum_{j=1}^n \bar{v}_j = n \bar{u}_i,
$$
since $\sum_j \bar{v}_j = \sum_l S_v^l = 0$.
We now have a system for $R_i(C)$ and $\bar{u}_i$:
$$
(n-\lambda)\bar{u}_i = K R_i(C) \quad \text{and} \quad (K-\lambda)R_i(C) = n \bar{u}_i.
$$
Multiplying the two gives $(n-\lambda)(K-\lambda)R_i(C) = nK R_i(C)$. Assuming $R_i(C) \neq 0$, we have $(n-\lambda)(K-\lambda) = nK$, which simplifies to $\lambda^2 - (n+K)\lambda = 0$. The non-zero root is $\lambda = n + K$.

Finally, we consider the subspace where global sums are zero ($S=0$) and marginals are zero ($R_i(C) = 0, C_j(C) = 0$).
\begin{enumerate}
    \item \textbf{Sub-case $u, v \neq 0$:} Since $R_i(C)=0$, Eq. \eqref{eq:eig_u} becomes $(n-\lambda)u_i^l = 0$. For a non-zero $u$, we must have $\lambda = n$. Consistency check: If $\lambda=n$, Eq. \eqref{eq:eig_C} becomes $(K-n)C_{ij} = \bar{u}_i + \bar{v}_j$. By choosing $C=0$ and ensuring $\bar{u}_i + \bar{v}_j = 0$ (which is possible while maintaining $u \neq 0$), $\lambda = n$ is established as an eigenvalue.
    \item \textbf{Sub-case $u=v=0, C \neq 0$:} If $u=v=0$, Eq. \eqref{eq:eig_C} becomes $(K-\lambda)C_{ij} = 0$. For a non-zero $C$, we must have $\lambda = K$. This is valid for any cost matrix $C$ with zero row/column sums (which exists for $n \ge 2$).
\end{enumerate}

Collecting all cases, we have proved that the non-zero eigenvalue must lie in the set $\{n, K, 2n, n+K, 2n+K\}$.  Now we explicitly construct a non-zero eigenvector $z = (u^1, \dots, u^K, v^1, \dots, v^K, C)$ for each value in this set. We assume $n, K \ge 2$ to ensure non-trivial subspaces exist. Let $\mathbf{1}_n$ denote the vector of all ones and $\mathbf{0}_n$ the zero vector.

\begin{itemize}
    \item \textbf{For $\lambda = 0$:}
    Let $C = 0$ and choose $u^l=\mathbf{1}_n, v^l = -\mathbf{1}_n$.
    \item \textbf{For $\lambda = K$:} 
     Let $u^l = \mathbf{0}_n$ and $v^l = \mathbf{0}_n$ for all $l$. Let $C$ be any non-zero matrix with zero row and column sums (e.g., $C_{11}=1, C_{12}=-1, C_{21}=-1, C_{22}=1$, and zeros elsewhere). 

    \item \textbf{For $\lambda = n$:} 
     Let $C = 0$. Let $v^l = \mathbf{0}_n$ for all $l$. Construct $u$ such that $\sum_i u_i^l = 0$ for all $l$ and $\sum_l u^l = 0$ . For example, set $u^1 = (1, -1, 0, \dots, 0)^T$, $u^2 = -u^1$, and $u^l = \mathbf{0}_n$ for $l > 2$.

    \item \textbf{For $\lambda = 2n$:} 
     Let $C = 0$. Set $u^l = v^l = \alpha_l \mathbf{1}_n$, where $\sum_l \alpha_l = 0$. For instance, let $\alpha_1 = 1, \alpha_2 = -1$, and $\alpha_l = 0$ for $l > 2$.

    \item \textbf{For $\lambda = n + K$:} 
     Let $u^l = \bar{u}$ and $v^l = \bar{u}$ for all $l$, where $\bar{u}$ is a vector satisfying $\sum_i \bar{u}_i = 0$ (e.g., $\bar{u} = (1, -1, 0, \dots, 0)^T$).
    Construct $C$ as $C_{ij} = -\frac{K}{n}(\bar{u}_i + \bar{u}_j)$. Note that $R_i(C) = \sum_j -\frac{K}{n}(\bar{u}_i + \bar{u}_j) = -\frac{K}{n}(n\bar{u}_i + 0) = -K\bar{u}_i$.

    \item \textbf{For $\lambda = 2n + K$:} 
    Let $u^l = \mathbf{1}_n$ and $v^l = \mathbf{1}_n$ for all $l$. Set $C$ to be a constant matrix $C_{ij} = -K/n$.
\end{itemize}

Since we have explicitly constructed non-zero eigenvectors for each candidate value, the spectrum of $A^T A$ is exactly $\{0, n, K, 2n, n+K, 2n+K\}$. This completes the proof.
\end{proof}

\section{Auxiliary Proofs for the Convergence Analysis}\label{app:convergence_proofs}

We use the notation and assumptions of Section~\ref{sec:PRAABCD} of the main article.

\subsection{Anderson-step bounds}\label{app:anderson_bound}

\begin{proof}[Proof of Lemma~\ref{lem:Anderson_bound}]
Let $\Gamma=\Gamma_k^*$, $R=R^k$, $Z=Z^k$, $r=r^k$, and $\delta=\delta_k$. Since $\Gamma$ minimizes the regularized least-squares problem, comparison with $\Gamma=0$ gives
\[
    \|r-R\Gamma\|^2+\delta\|Z\Gamma\|^2\leq\|r\|^2.
\]
In particular, $\|r-R\Gamma\|\leq\|r\|$ and $\|Z\Gamma\|\leq\delta^{-1/2}\|r\|$. The normal equation $R^T(R\Gamma-r)+\delta Z^TZ\Gamma=0$ additionally yields, after taking the inner product with $\Gamma$,
\[
    \|R\Gamma\|^2+\delta\|Z\Gamma\|^2
    =\langle R\Gamma,r\rangle,
\]
and hence $\|R\Gamma\|\leq\|r\|$. Therefore
\[
    \|M^kr^k\|
    =\alpha_k\|(Z+\beta_kR)\Gamma\|
    \leq\alpha_k(\delta_k^{-1/2}+\beta_k)\|r^k\|
    \leq\alpha_k\beta_k(1+\sqrt{c_2})\|r^k\|.
\]
Moreover, with $\widetilde r=r-R\Gamma$,
\[
    H^kr^k
    =\beta_k\big((1-\alpha_k)r+\alpha_k\widetilde r\big)
      -\alpha_kZ\Gamma.
\]
Using $0\leq\alpha_k\leq1$ and the preceding bounds,
\[
    \|H^kr^k\|
    \leq\beta_k\|r\|+\alpha_k\delta_k^{-1/2}\|r\|
    \leq\beta_k(1+\sqrt{c_2})\|r\|.
\]
Finally, $(1+\sqrt{c_2})^2\leq2(1+c_2)$, which proves \eqref{eq:H_bound}.
\end{proof}

\subsection{Localization and uniform curvature bounds}\label{app:localization}

\begin{proof}[Proof of Lemma~\ref{lem:localization}]
By Lemma~\ref{lem:compact_mod_lineality}, $\widehat{\mathcal{S}}_0:=\mathcal{S}_0\cap\mathcal{W}$ is compact and $\mathcal{S}_0=\widehat{\mathcal{S}}_0+\mathcal{N}$.

Consider one base BCD sweep from $z=(u,v,C)\in\mathcal{S}_0$, and set $y:=Az$. Since $J(z)=g(y)\leq\ell_0$, the vector $y$ ranges over the compact set $\mathcal{S}_{\ell_0}^g$. The $u$-block increment can be written in terms of $y$ as
\begin{equation}\label{eq:local_u_increment}
    u_i^{l,+}-u_i^l
    =\gamma\log\!\left(
       \frac{\mu_i^l}{\sum_{j=1}^n\exp(y_{lij}/\gamma)}
      \right).
\end{equation}
Hence $u^+-u$ is a continuous function of $y$ and is therefore uniformly bounded on $\mathcal{S}_{\ell_0}^g$. Moreover,
\[
    Az_u=y+(u^+-u)\oplus 0
\]
is also a continuous function of $y$, so $Az_u$ ranges over a compact set. Applying the same argument to the explicit $v$-update gives
\begin{equation}\label{eq:local_v_increment}
    v_j^{l,+}-v_j^l
    =\gamma\log\!\left(
      \frac{\nu_j^l}{\sum_{i=1}^n
      \exp((Az_u)_{lij}/\gamma)}
      \right),
\end{equation}
and therefore both $v^+-v$ and $Az_v$ are uniformly bounded as $z$ ranges over $\mathcal{S}_0$.

Since
\[
    \nabla_CJ(z_v)
    =\sum_{l=1}^K\hat X^l
      -\sum_{l=1}^K\exp((Az_v)^l/\gamma),
\]
continuity on the compact range of $Az_v$ yields a constant $G_C<\infty$ such that $\|\nabla_CJ(z_v)\|\leq G_C$. Using $C=\Proj_\mathcal{C}(C)$ and the nonexpansiveness of projection,
\[
    \|C^+-C\|
    \leq \eta\|\nabla_CJ(z_v)\|
    \leq \bar\eta G_C.
\]
Thus there exists a constant $R_0=R_0(\bar\eta)<\infty$, independent of the actual $\eta\in(0,\bar\eta]$, such that
\begin{equation}\label{eq:local_base_bounds}
    \|z_u-z\|\leq R_0,
    \qquad
    \|z_v-z\|\leq R_0,
    \qquad
    \|z^+-z\|\leq R_0,
    \qquad \forall z\in\mathcal{S}_0.
\end{equation}
In particular, $\|r\|\leq R_0$ for every base residual generated from a point in $\mathcal{S}_0$.

Let $c_H:=\sqrt{2(1+c_2)}$ and set
\[
    R:=\max\{R_0,c_HR_0\}.
\]
Define the compact convex enlargement of the canonical sublevel slice
\[
    \widehat{\mathcal{B}}
    :=\widehat{\mathcal{S}}_0+
      \{w\in\mathcal{W}:\|w\|\leq R\},
\]
and lift it back along the lineality directions:
\begin{equation}\label{eq:localized_set}
    \mathcal{B}:=(\widehat{\mathcal{B}}+\mathcal{N})\cap\mathcal{Z}.
\end{equation}
Since $\widehat{\mathcal{B}}\subseteq\mathcal{W}$ is compact and convex, $\mathcal{B}$ is convex, $\mathcal{B}+\mathcal{N}=\mathcal{B}$, and
\[
    \mathcal{B}\cap\mathcal{W}
    =\widehat{\mathcal{B}}\cap\mathcal{Z}
\]
is compact.

Now suppose $z^j\in\mathcal{S}_0$. Its canonical representative $\widehat z^j:=\Pi_\mathcal{W}z^j$ belongs to $\widehat{\mathcal{S}}_0$. For any $q\in\{z_u^j,z_v^j,z^{j+}\}$, \eqref{eq:local_base_bounds} gives
\[
    \|\Pi_\mathcal{W}q-\widehat z^j\|
    \leq\|q-z^j\|\leq R_0,
\]
so $\Pi_\mathcal{W}q\in\widehat{\mathcal{B}}$. Since $q\in\mathcal{Z}$ and
$q=\Pi_\mathcal{W}q+\Pi_\mathcal{N}q$, we obtain $q\in\mathcal{B}$. For the accelerated point, \eqref{eq:H_bound} and the nonexpansiveness of projection yield
\[
    \|z^{j+1}-z^j\|
    \leq\|H^jr^j\|
    \leq c_H\beta_j\|r^j\|
    \leq c_HR_0.
\]
The same projection argument gives $z^{j+1}\in\mathcal{B}$. Thus all points used at the first $k$ iterations are localized in the same set $\mathcal{B}$.

It remains to obtain uniform curvature constants. Every $z\in\mathcal{B}$ can be written uniquely as $z=w+d$ with $w=\Pi_\mathcal{W}z\in\mathcal{B}\cap\mathcal{W}$ and $d\in\mathcal{N}\subseteq\ker(A)$. Hence $Az=Aw$ and
\[
    \nabla^2J(z)=\frac1\gamma A^T\diag(\exp(Aw/\gamma))A.
\]
Since $\mathcal{B}\cap\mathcal{W}$ is compact, we may define
\[
    L_{\bar\eta}
    :=\max_{w\in\mathcal{B}\cap\mathcal{W}}
       \|\nabla^2J(w)\|<\infty,
\]
which makes $\nabla J$ $L_{\bar\eta}$-Lipschitz on the convex set $\mathcal{B}$. Likewise,
\[
    d_{\min}:=
    \min_{w\in\mathcal{B}\cap\mathcal{W}}
    \min_q\exp((Aw)_q/\gamma)>0,
\]
so the Hessians of the $u$- and $v$-block subproblems are bounded below by
\[
    \kappa_{\bar\eta}I,
    \qquad
    \kappa_{\bar\eta}:=\frac{nd_{\min}}{\gamma}>0.
\]
This completes the proof.
\end{proof}

\section{Detailed Experimental Setups}
\label{app:experiments}

This appendix details the data generation, preprocessing, baselines, and evaluation metrics used in the numerical experiments in Section 6.

\subsection{NYC Taxi Data Preprocessing and Baselines}
We utilized the NYC Taxi Trip Record dataset. The preprocessing pipeline involved:
\begin{enumerate}
    \item \textbf{Geospatial Discretization:} The region of interest (Manhattan and surroundings) was mapped to a $10 \times 10$ grid ($n=100$). Trip pickup and drop-off coordinates were mapped to these grid indices.
    \item \textbf{Scenario Definition:} We partitioned the data into $K=6$ scenarios based on specific time windows to capture temporal variability. The definitions are:
    \begin{itemize}
        \item \textbf{Scenario 1 (Weekday AM):} Mon-Fri, 07:00 -- 09:59.
        \item \textbf{Scenario 2 (Weekday Midday):} Mon-Fri, 10:00 -- 15:59.
        \item \textbf{Scenario 3 (Weekday PM):} Mon-Fri, 16:00 -- 18:59.
        \item \textbf{Scenario 4 (Weekday Night):} Mon-Fri, 19:00 -- 06:59 (next day).
        \item \textbf{Scenario 5 (Weekend Day):} Sat-Sun, 10:00 -- 18:59.
        \item \textbf{Scenario 6 (Weekend Night):} Sat-Sun, 19:00 -- 09:59 (next day).
    \end{itemize}
    \item \textbf{Training/Testing Split:} Data from months prior to June were used to construct the training plans $\{\hat{X}^l_{\text{train}}\}$. Data from June were used to construct the ground-truth test plans $\{\hat{X}^l_{\text{test}}\}$. All plans were normalized to sum to 1.
\end{enumerate}

The Gravity Model\cite{wilson1967statistical} serves as a parametric baseline. It models the flow $T_{ij}$ between zone $i$ and zone $j$ as:
\begin{equation}
    T_{ij} \propto O_i^\alpha D_j^\beta d_{ij}^{-\gamma},
\end{equation}
where $O_i$ is the total outflow from $i$, $D_j$ is the total inflow to $j$, and $d_{ij}$ is the Euclidean distance between zone centroids. We implemented this using a Negative Binomial Regression (a Generalized Linear Model) to handle over-dispersed count data. The parameters $\alpha, \beta, \gamma$ were fitted on the aggregated training data. During testing, the fitted model predicts flows based on the test set's marginals ($O_i, D_j$) and fixed distances.

We adopt three complementary metrics to assess the predictive performance of different models on the NYC taxi flow prediction task. Root Mean Square Error (RMSE) quantifies the absolute error between predicted and actual flows:

\begin{equation}
\text{RMSE}(X_{\text{pred}}, X_{\text{true}}) = \sqrt{\frac{1}{n^2} \sum_{i=1}^{n}\sum_{j=1}^{n} (X_{\text{pred},ij} - X_{\text{true},ij})^2},
\end{equation}

where $X_{\text{true}} \in \mathbb{R}^{n \times n}$ is the observed flow matrix and $X_{\text{pred}}$ is the predicted flow matrix. Lower RMSE indicates better overall agreement in absolute flow values.

Pearson Correlation Coefficient (Correlation) measures the linear association between the predicted and true flow patterns:

\begin{equation}
\text{Corr}(X_{\text{pred}}, X_{\text{true}}) = \frac{\text{Cov}(\text{vec}(X_{\text{pred}}), \text{vec}(X_{\text{true}}))}{\sigma_{\text{pred}} \sigma_{\text{true}}},
\end{equation}

where $\text{vec}(\cdot)$ denotes vectorization, $\text{Cov}(\cdot,\cdot)$ is covariance, and $\sigma$ denotes standard deviation. This metric, bounded within $[-1, 1]$, assesses whether the model captures the relative spatial patterns of flows, with higher values indicating better structural alignment.

Common Part of Commuters (CPC)\cite{lenormand2012universal} evaluates the overlap between predicted and observed flows, representing the fraction of commuters correctly assigned to origin--destination pairs. It is defined as

\begin{equation}
\text{CPC}(X_{\text{pred}}, X_{\text{true}}) = \frac{2 \sum_{i=1}^{n}\sum_{j=1}^{n} \min(X_{\text{pred},ij}, X_{\text{true},ij})}{\sum_{i=1}^{n}\sum_{j=1}^{n} X_{\text{pred},ij} + \sum_{i=1}^{n}\sum_{j=1}^{n} X_{\text{true},ij}}.
\end{equation}

The CPC metric, bounded within $[0, 1]$, directly measures the proportion of flow mass that is correctly allocated in the predicted transport plan. A CPC value of 1 indicates perfect overlap between predicted and observed flows.

These three metrics provide a comprehensive assessment: RMSE measures absolute error magnitude, Correlation evaluates structural similarity, and CPC quantifies the practical overlap in flow assignments—all crucial for transportation planning applications.

\subsection{E-commerce Data and Evaluation Metrics}

For data cleaning, transactions with nonpositive quantities (including cancellations), zero-valued records, or missing Customer IDs were removed. For each user, Recency was defined as days since the last purchase, Frequency as the total number of transactions, and Monetary as total spending. These three features were Z-score normalized. We applied K-Means clustering with $k=50$ to generate the user segments. The cluster centers served as the user feature matrix $G$. The Matrix Factorization baseline was implemented using the Alternating Least Squares (ALS) algorithm from the \texttt{implicit} Python library, with 20 factors and regularization parameter $0.1$.

To evaluate the recommendation quality, we ranked items for each user segment $i$ based on the predicted transport mass (probability) $X_{ij}$. Let $\mathcal{R}_i$ be the set of true relevant items (purchased by segment $i$ in the test month), and $\mathcal{P}_i@K$ be the top-$K$ items predicted by the model.
\begin{itemize}
    \item \textbf{Precision@K:} The fraction of recommended items that are relevant.
    \begin{equation}
        \text{Precision@K} = \frac{1}{m} \sum_{i=1}^m \frac{|\mathcal{P}_i@K \cap \mathcal{R}_i|}{K}.
    \end{equation}
    \item \textbf{Recall@K:} The fraction of relevant items that were successfully recommended.
    \begin{equation}
        \text{Recall@K} = \frac{1}{m} \sum_{i=1}^m \frac{|\mathcal{P}_i@K \cap \mathcal{R}_i|}{|\mathcal{R}_i|}.
    \end{equation}
\end{itemize}

\bibliographystyle{siamplain}
\bibliography{references}